\documentclass[reqno,11pt]{amsart}
\usepackage{amssymb}
\usepackage{amsmath}
\usepackage{stmaryrd}
\usepackage{graphicx}
\usepackage{color}
\usepackage{cite}
\usepackage[font=small]{caption}
\usepackage{float}
\usepackage{empheq}
\usepackage{bm}

\usepackage{tikz}
\usepackage{pgfplots}
\usepackage{wrapfig}
\pgfplotsset{compat=1.10}
\usepgfplotslibrary{fillbetween}
\usetikzlibrary{patterns}
\newtheorem{theorem}{Theorem}[section]
\newtheorem{corollary}[theorem]{Corollary}
\newtheorem{lemma}[theorem]{Lemma}
\newtheorem{proposition}[theorem]{Proposition}

\newtheorem{remark}[theorem]{Remark}

\newcommand{\rd}{\mathrm{d}}

\definecolor{cadmiumgreen}{rgb}{0.0, 0.42, 0.24}
\numberwithin{equation}{section}
\numberwithin{figure}{section}
\makeatletter
\@namedef{subjclassname@2020}{%
	\textup{2020} Mathematics Subject Classification}
\makeatother
\begin{document}
%%%%%%%%%%%%%%%%
%%%%%%%%%%%%%%%%

%%%%%%%%%%%%%%%%
\title[ ]{On a  Free Boundary Model for Three-Dimensional MEMS with a Hinged Top Plate II: Parabolic Case}
%%%%%%%%%%%%%%%%

%%%%%%%%%%%%%%%%

\author{Katerina Nik}
\address{Faculty of Mathematics\\ University of Vienna \\ Oskar-Morgenstern-Platz 1 \\ A--1090 Vienna\\ Austria}
\email{katerina.nik@univie.ac.at}
%%%%%%%%%%%%%%%%
%
%\thanks{}
%
\date{January 29, 2021}
\keywords{MEMS, free boundary problem, hinged plate, well-posedness, touchdown}
\subjclass[2020]{35K91, 35R35, 35M33, 35Q74, 35B44}

%35K91- Semilinear parabolic equations with Laplacian, bi-Laplacian or poly-Laplacian
%35R35- FBPs for PDEs 
%35M33- Initial-bvps for mixed-type systems of PDEs
%35Q74- PDEs in connection with mechanics of deformable solids
%35B44- Blow-up context of PDEs

%%%%%%%%%%%%%%%%
%%%%%%%%%%%%%%%%
\begin{abstract}
A parabolic free boundary problem modeling a three-dimensional electrostatic MEMS 
device is investigated. The device is made of a rigid ground plate and an elastic top plate 
which is hinged at its boundary, the plates being held at different voltages. The model couples a fourth-order semilinear parabolic equation for the deformation of the top plate to a Laplace equation for the electrostatic potential in the device. 
The strength of the coupling is tuned by a parameter $\lambda$ which is proportional 
to the square of the applied voltage difference.  It is proven that the model is locally 
well-posed in time and that, for $\lambda$ sufficiently small, solutions exist globally in time. In addition, touchdown of the top plate on the ground plate is shown to be the only possible finite time singularity. 
%Or: In addition, a criterion for global existence is obtained that excludes finite time singularities which are not physically relevant. 
\end{abstract}
%%%%%%%%%%%%%%%%
%%%%%%%%%%%%%%%%
%
\maketitle
%
%%%%%%%%%%%%%%%%
%%%%%%%%%%%%%%%%
\section{Introduction and main results}
%%%%%%%%%%%%%%%%
%%%%%%%%%%%%%%%%
Microelectromechanical systems (MEMS), which refer to tiny integrated devices combining electrical 
and mechanical elements, are essential parts of modern technology \cite{PeleskoBernstein2002, Younis2011}. In an idealized set-up, an electrostatically actuated MEMS device is built 
%or: The focus of this paper os the analysis of a free boundary model describing an idealized electrostatically actuated microelectromechanical system (MEMS). Such an idealized MEMS device consists 
of two thin conducting plates: a rigid ground plate above which an elastic plate is suspended, see Figure ~\ref{cross}. Holding the two plates at different voltages generates a Coulomb force across the device and induces a deformation of the elastic plate, thereby changing the geometry of the device and converting electrostatic to mechanical energy.

%%%%%%%%%%%%%%%%%%%%%%%%%%%%%%%%%%%%%%%%%%
%%%%%%%%%%%%%%%%
\begin{figure}
	\begin{raggedleft}
		\begin{tikzpicture}[scale=0.73]
			\draw[black, line width = 2pt] (-0.2,0)--(0.2,0);
			\node at (-0.27,0.3) {${\color{black} 0}$};
			\draw[black, line width = 2pt] (-7,0)--(-7,-5);
			\draw[black, line width = 2pt] (7,-5)--(7,0);
			\draw[black, line width = 2pt] (-7,-5)--(7,-5);
			\draw[blue, line width = 2pt] plot[domain=-7:7] (\x,{-1.3-1.7*cos((0.78*pi*\x/7) r)});
			\draw[black, line width = 1pt, arrows=->] (3,0)--(3,-2.05);
			\node at (3.4,-0.7) {${\color{black} u}$};
			\node[draw,rectangle,white,fill=white, rounded corners=5pt] at (2,-4.5) {$\Omega_1$};
			\node at (-4.5,-3.5) {${\color{black} \Omega(u)}$};
			\node at (5.7,-5.75) {${\color{black} \textbf{ground plate } D}$};
			\draw (3.55,-5.75) edge[->,bend left, line width = 1pt] (2.3,-5.1);
			\node at (9.4,-1.63) {${\color{black} \textbf{elastic plate} }$};
			\draw (7.7,-1.6) edge[->,bend left, line width = 1pt] (5.4,-0.88);
			
			\node at (-0.3,1.7) {$z$};
			\draw[black, line width = 1pt, arrows = ->,dashed] (0,-5.7)--(0,1.7);
			\node at (8.6,0.3) {$x_1$};
			\draw[black, line width = 1pt, arrows = ->,dashed] (-7.7,0)--(8.7,0);
			\node at (1.53,0.53) {$x_2$};
			\draw[black, line width = 1pt, arrows = ->,dashed] (0,0)--(1.7,1);
			
		\end{tikzpicture}
		\caption{Cross section of the idealized MEMS device.}\label{cross}
	\end{raggedleft}
\end{figure}
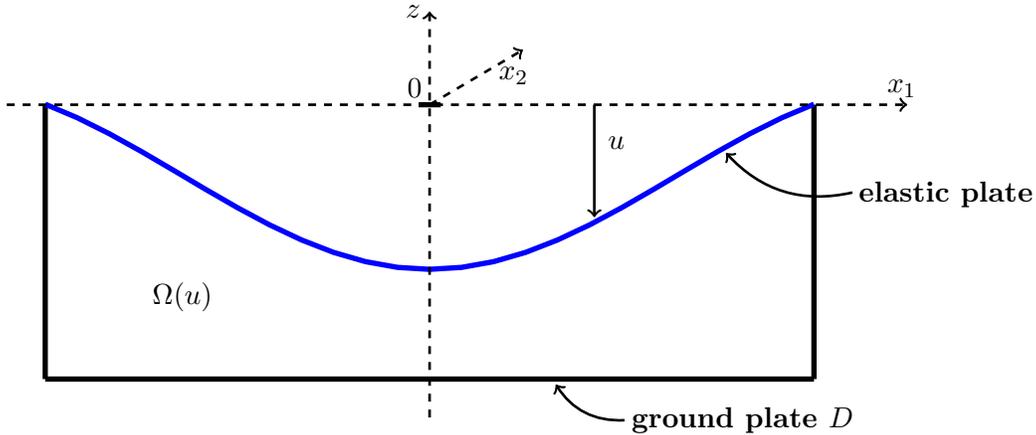
%%%%%%%%%%%%%%%%
%%%%%%%%%%%%%%%%%%%%%%%%%%%%%%%%%%%%%%%%%%%%%%%%%%%%%%%%%%%%%%%%%%%%%%%%%%%%%%%%%
In order to construct a mathematical model, we assume that the physical state of the device is fully described by the vertical deformation of the elastic plate from rest and the electrostatic potential between the two plates. We further assume that the shape of the ground plate and that of the elastic plate at rest are presented by $D \subset \mathbb{R}^2$.
%or: ...represented by $D$,  which is a bounded convex domain in $\mathbb{R}^2$ with $C^4$-boundary. 
After a suitable rescaling the ground plate is located at height $z=-1$ while the elastic plate at rest is located at $z=0$. Denoting the vertical deformation of the elastic plate at time $t>0$ and position 
$x=(x_1,x_2) \in D$ by $u=u(t,x) \in (-1,\infty)$, the evolution of $u$ is given in the damping dominated regime by 
\begin{align}
	\partial_t u &+ \beta \Delta^2 u - \tau \Delta u \nonumber
	\\%[-0.05cm]
	&= - \lambda \left( \varepsilon^2 \vert \nabla' \psi_{u(t)} (x,u(t,x)) \vert^2 + (\partial_z \psi_{u(t)}
	(x,u(t,x)))^2 \right), \quad x \in D, \; t>0,   \label{evolution1}
\end{align}
with hinged boundary conditions 
\begin{equation}
	u=\Delta u - (1 - \sigma) \kappa \partial_{\nu}u  = 0, \quad x \in \partial D, \; t>0,
	\label{evolution2}
\end{equation}
and initial condition
\begin{equation}
	u(0,x)=u^0(x), \quad  x \in D.
	\label{evolution3}
\end{equation}
The parameters $\beta>0$, $\sigma \in (-1,1)$, and $\tau \geq 0$ arise from the modeling of the mechanical forces and reflect bending, torsion, and stretching of the elastic plate, respectively. The 
right-hand side of \eqref{evolution1} is due to the electrostatic forces exerted on the elastic plate with parameter $\lambda >0$ proportional to the square of the applied voltage difference and the device's aspect ratio $\varepsilon>0$.  The boundary conditions \eqref{evolution2} mean that the elastic plate is hinged at its boundary. 
Here $\kappa$ is the curvature of the boundary $\partial D$ and $\nu$ the outward unit 
normal on $\partial D$. Finally, the electrostatic potential $\psi_{u(t)}=\psi_{u(t)}(x,z)$ is defined for 
$t>0$ and $(x,z) \in \Omega(u(t))$, where 
\begin{equation*}
	%\label{geometrydevice}
	\Omega(u(t)):= \{ (x,z)\,: \,  x\in D, \, -1<z<u(t) \}
\end{equation*}
is the three-dimensional cylinder between the rigid ground plate at $z=-1$ and the deformed elastic plate at $z=u(t)$. For each $t>0$, $\psi_{u(t)}$ satisfies the rescaled Laplace equation 
\begin{equation}
	\varepsilon^2 \Delta' \psi_{u(t)} + \partial_z^2 \psi_{u(t)}  =0,   \quad (x,z) \in \Omega(u(t)), \; t>0,
	\label{evolution4}
\end{equation}
with nonhomogeneous Dirichlet boundary conditions 
\begin{equation}
	\psi_{u(t)} (x,z) = \frac{1+z}{1+u(t,x)}, \quad (x,z) \in \partial  \Omega(u(t)),\; t>0.
	\label{evolution5}
\end{equation}
Here $\nabla'$ and $\Delta'$ are respectively the gradient and Laplace operator with respect to $x\in D$ for functions of $x$ and $z$. Note that \eqref{evolution1} is a nonlocal semilinear fourth-order parabolic equation for the plate deformation $u$, which is coupled to the ellliptic equation \eqref{evolution4} in the free domain $\Omega(u(t))$ for the electrostatic potential $\psi_{u(t)}$.  
For more details on the derivation of the free boundary problem  \eqref{evolution1}-\eqref{evolution5}, see ~\cite[Section 2]{N20}. 

A crucial feature of \eqref{evolution1}-\eqref{evolution5} is that it is only meaningful as long as the deformation $u$ satisfies $u>-1$.  When $u$ attains the value  $-1$ somewhere in $D$ at some time $T_*>0$, i.e. when 
\begin{equation}
	\label{tp}
\lim_{t \rightarrow T_*} \min_{x \in D} u(t,x)=-1,
\end{equation}
this respresents a touchdown of the elastic plate on the ground plate. This physical phenomenon has 
been observed experimentally in MEMS devices for sufficiently large voltage values $\lambda$ and is often called pull-in instability \cite{MarquesCastelloShkel2005, PeleskoBernstein2002}. It is characterized by the existence of a threshold value $\lambda_*$ of $\lambda$, such that touchdown occurs in finite time for $\lambda > \lambda_*$ and never occurs for $0< \lambda < \lambda_*$. An interesting and important question is whether \eqref{evolution1}-\eqref{evolution5} exhibits such a touchdown behavior, i.e. whether \eqref{tp} takes place.  
%It is an interesting and important question whether \eqref{evolution1}-\eqref{evolution5} exhibits such a touchdown behavior, i.e. whether \eqref{tp} occurs. 

%The main novelties of this paper are the treatment of the hinged boundary conditions \eqref{evolution2} and of the three-dimensional geometry of the device $\Omega(u(t))$. 
Free boundary problems similar to \eqref{evolution1}-\eqref{evolution5} where the hinged boundary conditions \eqref{evolution2} are replaced by the clamped boundary conditions $u= \partial_{\nu} u=0$ on $\partial D$, have been investigated in ~\cite{LaurencotWalker2016, LaurencotWalker2018} and in ~\cite{LaurencotWalkerI,LaurencotWalkerESAIM, LaurencotWalkerwave} for the one-dimensional setting ~$D=(-1,1)$. The stationary version of \eqref{evolution1}-\eqref{evolution5} has already been studied in \cite{N20}. However, the parabolic problem \eqref{evolution1}-\eqref{evolution5} has not been discussed in the literature so far. In this paper we refine and extend some of the arguments used in  \cite{LaurencotWalkerI, LaurencotWalker2016, LaurencotWalker2018} to obtain results on solutions of \eqref{evolution1}-\eqref{evolution5} where we have to cope with the hinged boundary conditions \eqref{evolution2} and the three-dimensional geometry of ~$\Omega(u(t))$.

We now state our main findings on \eqref{evolution1}-\eqref{evolution5}. From now on, we assume that 
\begin{equation*}
	D \text{ is a bounded convex domain in } \mathbb{R}^2 \text{ with } C^4 \text{ boundary }
\end{equation*}
and that the parameters
\begin{equation*}
	\varepsilon >0, \quad \beta >0, \quad \tau \geq 0, \quad \sigma \in (-1,1)
\end{equation*}
are fixed. We will state additional assumptions on $D$ whenever needed. For further use, given $q\in [1,\infty]$, we define the Sobolev space $W^{s}_{q,\mathcal{B}}(D)$ 
incorporating the boundary conditions \eqref{evolution2}, if meaningful, by 
\begin{equation*}
	W^{s}_{q,\mathcal{B}}(D):= 
	\left\{\begin{array}{ll}
		\vspace{0.17cm}
		W^{s}_{q}(D),  &  s \in [0, \tfrac{1}{q}], 
		\\
		\vspace{0.17cm}
		\left\{ v\in  W^{s}_{q}(D)\, : \, v=0 \text{ on } \partial D \right\},
		& s \in \big(\tfrac{1}{q}, 2+\tfrac{1}{q}\big], 
		\\
		\left\{v\in  W^{s}_{q}(D)\, : \, v = \Delta v -(1 - \sigma) \kappa 
		\partial_{\nu}v=0 \text{ on } \partial D\right \},
		& s \in \big(2+\tfrac{1}{q}, 4\big].
	\end{array} \right.
\end{equation*} 

The following result shows that \eqref{evolution1}-\eqref{evolution5} is locally well-posed for any $\lambda$ and globally well-posed for small $\lambda$ provided that $u^0$ is small as well. 
 
\begin{theorem}
\label{wellposedness}\hspace{-0.1cm}\textnormal{(Well-posedness)} 
	Given $4\xi \in (\frac{7}{3},4)\backslash\{\frac{5}{2}\}$, consider an initial 
value  ~$u^{0}\in W^{4\xi}_{2,\mathcal{B}}(D)$ such that $u^{0}>-1$ in $D$. Then, the following hold. 

\begin{enumerate}
\item[\textit{(i)}]\textnormal{(Local existence)}  For each voltage value $\lambda>0$, there is a unique solution $(u,\psi_u)$ to 
\eqref{evolution1}-\eqref{evolution5} on the maximal interval of existence $[0,T_m)$ in the sense that
\begin{equation}
	\label{regularity.1}
	u \in C^1((0,T_m),L_2(D)) \cap  C((0,T_m),W^{4}_{2,\mathcal{B}}(D)) \cap C([0,T_m),W^{4\xi}_{2,\mathcal{B}}(D))
\end{equation}
satisfies \eqref{evolution1}-\eqref{evolution3} together with 
\begin{equation*}
	u(t,x) > -1, \quad (t,x) \in [0,T_m) \times D,
\end{equation*}
and 
\begin{equation}
	\label{regularity.2}
	\psi_{u(t)}\in W^2_2(\Omega(u(t)))
\end{equation}
solves \eqref{evolution4}-\eqref{evolution5} in $\Omega(u(t))$ for each $t \in [0,T_m)$.  In addition, if $\xi=1$, then 
$u \in  C^1([0,T_m),L_2(D)) \cap  C([0,T_m),W^{4}_{2,\mathcal{B}}(D))$. \\
		
\item[\textit{(ii)}]\textnormal{(Norm blow-up or touchdown)} If, for each $T>0$, there is $\rho(T) \in (0,1)$ such that 
\begin{equation*}
	\Vert u(t) \Vert_{W^{4\xi}_{2,\mathcal{B}}(D)} \leq \rho(T)^{-1} \; \text{ and } \; u(t) \geq -1+\rho(T)  \text{ in } D
\end{equation*}
for $t\in [0,T_m)\cap [0,T]$, then the solution exists globally in time, i.e. $T_m=\infty$. 
\\
		
\item[\textit{(iii)}]\textnormal{(Global existence)} Given $\rho \in \left(0,\frac{1}{2}\right)$, there are ~$\lambda_*:= \lambda_*(\rho, \varepsilon) > 0$ and ~$N:=N(\rho,\varepsilon) > 0$ such that $T_m = \infty$  provided that $\lambda \in (0,\lambda_* )$, $u^0 \geq -1+2 \rho$ in $D$, and 
\begin{equation*}
	\Vert u^0 \Vert _{W^{4\xi}_{2, \mathcal{B}}(D)} \leq N. 
\end{equation*}
In this case, $u \in L_{\infty}(0,\infty; W^{4\xi}_{2,\mathcal{B}}(D))$ with 
\begin{equation*}
	\inf_{(t,x)\in [0,\infty)\times D }u(t,x) >  -1.
\end{equation*}
\end{enumerate}
\end{theorem}

The proof of Theorem \ref{wellposedness} is performed in Section \ref{wp}. We first transform the Laplace equation \eqref{evolution4}-\eqref{evolution5} to a fixed cylinder which leads to an 
elliptic boundary value problem with nonconstant coefficients depending on $u$, $\nabla u$ and $\Delta u$. Solving this transformed problem for a given $u$ enables us to formulate the full free boundary problem as a nonlocal semilinear evolution equation for $u$. We then employ semigroup theory and a fixed point argument to solve this evolution equation. 

Note that part (iii) of Theorem \ref{wellposedness} provides uniform estimates on $u$ in the $W^{4\xi}_{2,\mathcal{B}}$-norm and implies that touchdown of $u$ on $-1$ does not even take place in infinite time. Also note that part (ii) of Theorem \ref{wellposedness} ensures that, if the maximal existence time $T_m$ is finite, then
\begin{equation*}
	\limsup_{t \rightarrow T_m} \: \Vert u(t)\Vert_{
		W^{4\xi}_{2,\mathcal{B}}(D)}=\infty 
	\;  
	\text{ or }\; 
	\liminf_{t \rightarrow T_m} \: 
	\min_{ x \in \overline{D}} u(t,x)=-1.
\end{equation*}
From a physical viewpoint, this outcome is not yet satisfactory since it does not imply that a finite time singularity is only due to the touchdown phenomenon \eqref{tp} described above. \\
\\
The next result improves part (ii) of Theorem \ref{wellposedness} showing that, in fact, only the touchdown phenomenon \eqref{tp} can generate a finite time singularity. 

\begin{theorem}\hspace{-0.1cm}\textnormal{(Touchdown)}
\label{globalcriterion}
Let $D$ be a bounded convex domain in $ \mathbb{R}^2$ with $C^{4,\gamma}$  boundary for 
some $\gamma \in (0,1)$. Under the assumptions of Theorem \ref{wellposedness}, let $(u,\psi_u)$ be the unique solution to \eqref{evolution1}-\eqref{evolution5} on  the maximal interval of existence $[0,T_m)$.  
Assume that there are $T_0>0$ and $\rho_0\in (0,1)$ such that 
\begin{equation}
	\label{singul1}
	u(t)\geq -1+\rho_0 \; \text{ in } \; D, \; t \in [0,T_m)\cap [0,T_0].
\end{equation}
Then $T_m \geq T_0$. 

Moreover, if, for each $T>0$, there is $\rho(T) \in (0,1)$ such that 
\begin{equation*}
	u(t)\geq -1+\rho(T) \; \text{ in } \; D, \; t \in [0,T_m)\cap [0,T],
\end{equation*}
then $T_m=\infty$.
\end{theorem}

We point out that it remains an open problem whether a finite time singularity occurs when $\lambda$ is large enough. 
The proof of Theorem \ref{globalcriterion} is given in Section \ref{rc} and relies on semigroup theory in (negative) Besov spaces and on the gradient flow structure of  \eqref{evolution1}-\eqref{evolution5}, the latter being inherent in the model derivation (see \cite{N20}). 
%or.... gradient flow structure of  \eqref{evolution1}-\eqref{evolution5}, which is inherent in the model derivation \cite{N20}.  
Indeed, introducing the rescaled total energy of the device
\begin{equation*}
	E(u):= E_m(u) - \lambda E_e (u)
\end{equation*}
involving the mechanical energy 
\begin{equation}
\label{energym}
E_m(u):= \beta \int_D \bigg( \frac{1}{2} (\Delta u)^2 +
(1-\sigma) \big( (\partial_{x_2}\partial_{x_1} u)^2 - \partial_{x_1}^2 u \, \partial_{x_2}^2 u\big) \bigg) \, \rd x
+ \frac{\tau}{2} \int_D \vert \nabla u \vert^2 \, \rd x
\end{equation}
and the electrostatic energy 
\begin{equation}
\label{energye}
E_e (u) := \int_{\Omega (u)} \Big( \varepsilon^2 \vert \nabla'\psi_u \vert^2 +  (\partial_z \psi_u)^2 \Big) \, \rd(x,z), 
\end{equation}
the following energy equality holds. 

\begin{theorem} \hspace{-0.1cm}\textnormal{(Energy equality)}
	\label{Tenergyequality}
	Under the assumptions of Theorem \ref{wellposedness},
	\begin{equation}
		\label{energyequality}
		E(u(t)) + \int_0^t \Vert \partial_t u(s)\Vert_{L_2(D)}^2 \, \rd s = E(u^0)
	\end{equation}
for $t \in [0,T_m)$.
\end{theorem}

The proof of Theorem ~\ref{Tenergyequality} is given in Section \ref{ee}. 
One of the difficulties in the proof of Theorem ~\ref{Tenergyequality} is the computation of 
the derivative $\frac{\rd}{\rd t} E_e(u(t))$, 
which is due to the fact that the underlying domain $\Omega(u(t))$ varies with respect to $u(t)$. 
Another difficulty is that the time regularity of $u$ as stated in part (i) of Theorem \ref{wellposedness} is not sufficient for a direct computation of the 
derivative $\frac{\rd}{\rd t} E(u(t))$ 
and an approximation argument has to be used.

%%%%%%%%%%%%%%%%
%%%%%%%%%%%%%%%%
\section{Auxiliary results}
\label{aux}
%%%%%%%%%%%%%%%%
%%%%%%%%%%%%%%%%
We shall first %collect 
derive properties of solutions to the Laplace equation \eqref{evolution4}-\eqref{evolution5} in dependence of a given deformation $u(t): \overline{D} \rightarrow (-1,\infty)$ for a fixed time $t$. In order to do so, we transform  \eqref{evolution4}-\eqref{evolution5} to the fixed cylinder $\Omega:=D\times (0,1)$. 

More precisely, given  ~$q \in [3,\infty]$ and an arbitrary function $v\in W^2_{q,\mathcal{B}}(D)$ taking values in  ~$(-1,\infty)$, we define the diffeomorphism 
$T_v:\overline{\Omega(v)}\rightarrow \overline{\Omega}$ by 
\begin{equation}
		\label{diffeom}
	T_v(x,z):=\left(x,\frac{1+z}{1+v(x)}\right),\quad (x,z)\in \overline{\Omega(v)}.
\end{equation}
The inverse of $T_v$ is given by 
\begin{equation*}
	T_v^{-1}(x,\eta)=\left(x,(1+v(x))\eta -1\right), \quad (x,\eta) \in \overline{\Omega},
\end{equation*}
and the Laplace operator $\varepsilon^2 \Delta' + \partial_z^2$ is transformed to the $v$-dependent differential operator 
\begin{align}
	\label{opL}
	\mathcal{L}_{v} w &:= \varepsilon^{2} \Delta' w  -  2 \varepsilon^{2} \eta  
	\frac{\nabla v(x)}{1+v(x)}\cdot \nabla' \partial_{\eta} w
	+  \frac{1+\varepsilon^{2} \eta^2 
		\vert \nabla v(x)\vert^2}{(1+v(x))^2} \partial_{\eta}^2 w 
	\nonumber \\%[0.1cm]
	&\quad + \varepsilon^{2} \eta \left( 2 
	\frac{\vert \nabla v(x)\vert^2}{(1+v(x))^2} 
	- \frac{\Delta v(x)}{1+v(x)}\right) \partial_{\eta} w.
\end{align}
Setting $\phi_v:= \psi_v \circ T_v^{-1}$,  \eqref{evolution4}-\eqref{evolution5} is  equivalent to 
\begin{empheq}{align}
	( \mathcal{L}_{v} \phi_v)(x,\eta) &=0,  \qquad (x,\eta) \in \Omega,
	\label{ellsub1} \\%[-0.05cm] 
	\phi_v(x,\eta)  &= \eta,  \qquad (x,\eta) \in \partial  \Omega.
	\label{ellsub2}
\end{empheq}
Note that \eqref{ellsub1}-\eqref{ellsub2} is an elliptic boundary value problem with nonconstant coefficients, but in the fixed domain $\Omega$. 
We further define, for $\rho \in (0,1)$, the open subset 
\begin{equation}
	\label{solutionset}
	S_q (\rho):= \left\{ v \in  W^{2}_{q,\mathcal{B}}(D)\,: \, 	\Vert v \Vert_{W^{2}_{q}(D)} < \frac{1}{\rho} \;\text{ and }\;
	v(x) > -1 + \rho \; \text{ for }\, x \in D \right\} 
\end{equation}
of $W^{2}_{q,\mathcal{B}}(D)$ with closure 
\begin{equation*}
	\overline{S}_q (\rho)=\left\{ v \in  W^{2}_{q,\mathcal{B}}(D)\,: \, \Vert v \Vert_{W^{2}_{q}(D)} 
	\leq  \frac{1}{\rho}  \;\text{ and }\; v(x)
	\geq  -1 + \rho  \; \text{ for }\, x \in D \right\} .
\end{equation*}
The following proposition provides %collects 
some important properties with respect to $v$ of the solution $\phi_v$ to  \eqref{ellsub1}-\eqref{ellsub2}. 

\begin{proposition}
	\label{ellipticp}
For each $v\in \overline{S}_q (\rho)$ there is a unique solution $\phi_v \in W^2_2(\Omega)$ to \eqref{ellsub1}-\eqref{ellsub2} 
and 
\begin{equation*}
	\Vert \phi_v \Vert_{W^2_2(\Omega)} \leq c_0, \quad v \in  \overline{S}_q (\rho),
\end{equation*}
with 
\begin{equation*}
	\Vert \phi_{v_1} -   \phi_{v_2} \Vert_{W^2_2(\Omega)} \leq c_0 \Vert v_1 - v_2 \Vert_{W^2_q(D)} , \quad v_1,v_2 \in  \overline{S}_q (\rho),
\end{equation*}
for some positive constant $c_0$ depending only $q$, $\rho$, $\varepsilon$, and $D$. Furthermore, the mapping 
\begin{equation*}
	g: S_q(\rho) \rightarrow L_2(D), \quad v \mapsto \frac{1 + \varepsilon^{2}  
		\vert \nabla v\vert^2}{(1+v)^2}  \left( \partial_{\eta} \phi_v(\cdot,1)
	\right)^2
\end{equation*}
is analytic, bounded, and globally Lipschitz continuous. 
\end{proposition}

\begin{proof}
This follows from \cite[Proposition 2.1 \& Equation (2.26)]{LaurencotWalker2016}, because the proof 
therein only uses the boundary condition $v=0$ on $\partial D$. 
%....and the proof therein holds also in our case/setting. 
\end{proof}

For a time-dependent function $u(t) \in \overline{S}_q (\rho)$, Proposition \ref{ellipticp} and the just introduced notation 
reveal that the potential $\psi_{u(t)}$ belongs to $W^2_2(\Omega(u(t)))$ and solves  \eqref{evolution4}-\eqref{evolution5}. Concerning the right-hand side of equation \eqref{evolution1}, we have the relation
\begin{align*}
	- \lambda  \big(
	\varepsilon^2 
	&\vert \nabla'\psi_{u(t)}(\cdot, u(t))\vert^2 +  
	(\partial_z \psi_{u(t)}(\cdot,u(t)))^2 \big)
	\nonumber\\
	&\qquad = 	- \lambda   \frac{1+ \varepsilon^2 \vert \nabla u(t) \vert^2}{(1+u(t))^2} 
	(\partial_{\eta} \phi_{u(t)}(\cdot,1))^2 = - \lambda g(u(t)),
\end{align*}
where we used that $\nabla' \phi_{u(t)} (x,1) = (0,0)$ for $x \in D$ due to $\phi_{u(t)}(x,1)= 1$ by \eqref{ellsub2}. This puts us in a position to formulate \eqref{evolution1}-\eqref{evolution5} as a single nonlocal evolution 
equation only involving the deformation $u$, see \eqref{CauchyProblem} below. \\
\\
We next show a poposition which will be used in the proof of Theorem \ref{Tenergyequality}.

\begin{proposition}
	\label{electrostaticE}
	Let $T>0$, $4\xi \in (\frac{7}{3},4)\backslash\{\frac{5}{2}\}$, 
	%$4\xi \in (7/3,4)\backslash\{ 5/2\}$ with $4\xi \neq 5/2$, 
	and let 
	\begin{equation*}
		v \in C^1([0,T],W^{4\xi}_{2,\mathcal{B}}(D))
	\end{equation*}
	be such that $v(t,x) > -1$ for $(t,x)\in [0,T]\times \overline{D}$. 
	Then
	\begin{equation}
		\label{electrostaticE1}
		E_e(v(t_2)) - E_e(v(t_1)) 
		= - \int_{t_1}^{t_2} \int_D 
		g(v(s)) \partial_t v(s) \, \rd x \rd s
	\end{equation}
	for $ 0\leq t_1\leq t_2 \leq T$. 
\end{proposition}

When $D=(-1,1)$ is one-dimensional and clamped boundary conditions $v(t,\pm 1) 
= \partial_x v(t,\pm 1) =0$ are considered, this proposition has been proved in 
\cite[Proposition 2.2]{LaurencotWalkerI}. Our proof follows a similar spirit: 
We first rewrite the electrostatic energy $E_e(v(t))$ as an integral over the 
fixed domain $\Omega$. The resulting electrostatic energy is thus expressed 
in terms of $\phi_{v(t)}$. We next verify the differentiability of $\phi_v$ with respect to $t$. Using this we then show that the transformed electrostatic energy is differentiable with respect to $t$ and compute its derivative. Finally we transform the obtained derivative back to the original coordinates to get 
\eqref{electrostaticE1}. 

\begin{proof}
	Since $W^{4\xi}_2(D)$ embeds continuously in $W^2_3(D)$ and in $C(\overline{D})$, there 
	is $\rho \in (0,1)$ so that $v(t) \in \overline{S}_3 (\rho)$ for all $t \in [0,T]$, and 
	hence Proposition \ref{ellipticp} can be applied. To simplify notation let, for each 
	$t\in [0,T]$, $\phi(t)=\phi_{v(t)} \in W^2_2(\Omega)$ denote the solution to \eqref{ellsub1}-\eqref{ellsub2}
	associated to $v(t)$ and $\psi(t)=\psi_{v(t)} \in W^2_2(\Omega(v(t)))$ denote the corresponding 
	solution to \eqref{evolution4}-\eqref{evolution5} also associated to $v(t)$. 

	For $(t,x,\eta) \in [0,T]\times \Omega$, we set
	\begin{equation}
		\label{electrostatic_prope1}
		\Phi(t,x,\eta):= \phi(t,x,\eta) - \eta, \quad V(t,x):= 
		\frac{\nabla v(t,x)}{1+v(t,x)}\,,
	\end{equation}
and denote the components of $V$ by $V_i$, $i=1,2$.  
Recalling that the electrostatic energy $E_e$ is defined in \eqref{energye} and that $\psi(t)=\phi(t)\circ T_{v(t)}$ with the transformation $T_{v(t)}$ as in \eqref{diffeom}, 
we obtain, by the change of variables $(x,z)\rightarrow (x,\eta)$, 
\begin{equation}
	E_e(v(t))
	= \varepsilon^2 \int_{\Omega} 
	\vert \nabla' \phi(t) - \eta  V(t) \partial_{\eta} \phi(t)  \vert^2 \, (1+v(t))\, \rd(x,\eta) 
	+ \int_{\Omega} \frac{( \partial_{\eta} \phi(t))^2}{1+v(t)}\, \rd(x,\eta).
	\label{electrostatictr}
\end{equation}
Next, setting 
\begin{equation*}
	W^2_{2,\mathcal{B}}(\Omega):=\left\{ w \in W^2_2(\Omega)\,: \, w =0 \; \text{ on } \, \partial \Omega\right\},
\end{equation*}
we have by \eqref{electrostatic_prope1} that, for $t\in [0,T]$, the function $\Phi(t)\in W^2_{2,\mathcal{B}}(\Omega)$ solves  
\begin{empheq}{align*}
	-\mathcal{L}_{v(t)} \Phi(t) &=f(t)  \qquad  \text{ in } \Omega,
	 \\%[-0.05cm] 
	\Phi(t) &= 0  \qquad  \text{ on } \partial  \Omega
\end{empheq}
with 
\begin{equation*}
	f(t,x,\eta) := \varepsilon^2 \eta \, \big( \vert V(t,x)\vert^2
	- \text{div }V(t,x) \big), \quad (t,x,\eta) \in [0,T]\times\Omega,
\end{equation*}
satisfying $f(t) \in L_2(\Omega)$. For later use, we write  the operator $\mathcal{L}_{v(t)}$ in divergence form
\begin{equation*}
	\mathcal{L}_{v(t)} w= \text{div}(\bm{\alpha}(t)\nabla w ) 
	+\bm{b}(t) \cdot \nabla w
\end{equation*}
with
\begin{equation*}
	\bm{\alpha}(t):=
	\begin{pmatrix}
		\alpha_{1}(t) & 0 & \frac{\alpha_{2}(t)}{2} \\[0.09cm]
		0 & \alpha_{1}(t) &\frac{\alpha_{3}(t)}{2}  \\[0.09cm]
		\frac{\alpha_{2}(t)}{2}  & \frac{\alpha_{3}(t)}{2}   & \alpha_{4}(t)
	\end{pmatrix}
\end{equation*}
and $\bm{b}(t) := (b_1(t), b_2(t),b_3(t))$, where
\begin{align*}
	&\alpha_{1}(t,x,\eta):= \varepsilon^2, 
	\quad  \alpha_{2}(t,x,\eta):= - 2\varepsilon^2 \eta V_1(t,x), 
	\quad \alpha_{3}(t,x,\eta):=- 2\varepsilon^2 \eta V_2(t,x),
	\nonumber \\[0.05cm]
	&\alpha_{4}(t,x,\eta):= \frac{1}{(1+v(t,x))^2}+ \varepsilon^2 
	\eta^2  \vert V(t,x)\vert^2,
	\quad b_1(t,x,\eta):=\varepsilon^2 V_1(t,x),
	\nonumber \\[0.05cm]
	&b_2(t,x,\eta):=\varepsilon^2 V_2(t,x), 
	\quad b_3(t,x,\eta):= -\varepsilon^2 \eta  \vert V(t,x)\vert^2
	%\label{vecdiv}
\end{align*}
for $(t,x,\eta)\in [0,T]\times\Omega$. We now prove the differentiability of $\Phi$ with respect to $t$ and start by 
briefly recalling some properties of $\mathcal{L}_{v(t)}$; for further details and
proofs see ~\cite[Section 2]{LaurencotWalker2016}. Let us define a bounded linear operator $\mathcal{A}(t) \in 
\mathcal{L}(W^2_{2,\mathcal{B}}(\Omega), L_2(\Omega))$  by 
\begin{equation*}
	\mathcal{A}(t) w:= -\mathcal{L}_{v(t)} w, \quad w\in W^2_{2,\mathcal{B}}(\Omega), t \in [0,T].
\end{equation*}
For each $t \in [0,T]$, it ifollows from \cite[Proposition 2.7]{LaurencotWalker2016} that $\mathcal{A}(t)$ is invertible and 
$\Phi(t)=\mathcal{A}(t)^{-1} f(t)$. Using the time regularity of $v$ and the continuous embeddings 
\begin{equation*}
	W^{4\xi}_2(D) \hookrightarrow W^2_3(D) \hookrightarrow C^1(\overline{D})
\end{equation*}
we obtain by direct computation that $\alpha_{2}$, $\alpha_{3}$, $\alpha_{4}$, and $b_{3}$ belong to 
$ C^1([0,T],L_{\infty}(\Omega))$ and that $\Delta v/ (1+v)$ belongs to $C^1([0,T],L_3(D))$, so that 
%\begin{equation*}   
%	\alpha_{2}, \alpha_{3}, \alpha_{4}, b_{3} \in C^1([0,T],
%	L_{\infty}(\Omega)), \quad 
%	\frac{\Delta v}{1+v} \in C^1([0,T],L_3(D)),
%\end{equation*}
\begin{equation}
	\label{electrostatic_regu1}
	\mathcal{A}\in C^1([0,T],\mathcal{L}(W^2_{2,\mathcal{B}}(\Omega), L_2(\Omega))) \; \text{ and } \;  
	f\in C^1([0,T], L_2(\Omega)).
\end{equation}
Since the map taking an invertible operator to its inverse is continuously differentiable on the 
space of bounded operators, the mapping 
\begin{equation*}
	[t\mapsto \mathcal{A}(t)^{-1}]:[0,T]
	\rightarrow \mathcal{L}(L_2(\Omega),W^2_{2,\mathcal{B}}(\Omega))
\end{equation*} 
is continuously differentiable and hence $\Vert \mathcal{A}(t)^{-1} 
\Vert_{\mathcal{L}(L_2(\Omega), W^2_{2,\mathcal{B}}(\Omega))}\leq  \text{const}$. This, together with \eqref{electrostatic_regu1}, implies that
\begin{equation*}
	\label{electrostatic_regu3}
	\Phi \in C^1([0,T],W^2_{2,\mathcal{B}}(\Omega))
\end{equation*}
with derivative 
\begin{equation*}
	\partial_t \Phi(t)= \mathcal{A}(t)^{-1} \big( \partial_t f(t) - \partial_t 
	\mathcal{A}(t) \, \Phi(t) \big)\in W^2_{2,\mathcal{B}}(\Omega), 
	\quad t\in [0,T].
\end{equation*}
Therefore, in view of \eqref{electrostatic_prope1}, we get 
\begin{equation}
	\label{electrostatic_prope2}
	\phi \in C^1([0,T],W^2_2(\Omega))  \; \text{ with } \;   \partial_t \phi(t) 
	= \partial_t \Phi(t), \; t \in [0,T].
\end{equation}
We next consider the transformed electrostatic energy  \eqref{electrostatictr}. By direct calculations, we deduce from 
the time regularity of $\phi$ and $v$ that 
$E_e(v) \in C^1([0,T])$ with derivative 
\begin{align}
	\label{electrostatic_prope3}
	\frac{\rd}{\rd t} E_e(v(t))
	& = 2\varepsilon^2 \int_{\Omega} \Big[ 
	(\nabla' \phi (t) - \eta V(t)\partial_{\eta}\phi(t))\cdot
	(\nabla' \partial_t \phi(t) - \eta \partial_t V(t) \partial_{\eta} \phi(t) 
	\nonumber \\[-0.1cm]
	&\hspace{6.2cm}- \eta V(t) \partial_{\eta}\partial_t \phi(t))\Big] (1+v(t))\, 
	\rd(x,\eta) 
	\nonumber \\%[0.1cm]
	& \quad +  \varepsilon^2 \int_{\Omega} \vert \nabla' \phi(t)- \eta V(t)
	\partial_{\eta}\phi(t)\vert^2 \,\partial_t v(t)\, \rd(x,\eta)
	\nonumber\\%[0.1cm]
	& \quad + 2 \int_{\Omega} \frac{\partial_{\eta} \phi(t) \partial_{\eta}\partial_t \phi(t)}
	{1+v(t)}\, \rd(x,\eta) 
	- \int_{\Omega} (\partial_{\eta} \phi(t))^2  \frac{ \partial_t v(t)}{
		(1+v(t))^2} \, \rd(x,\eta) 
\end{align}
for $t \in [0,T] $. Now we write the right-hand side of \eqref{electrostatic_prope3} in a simpler form, and to do this  
we multiply the equation $\mathcal{L}_{v(t)}\phi(t)=0$ by 
$(1+v(t))\partial_t \phi (t)$ and integrate over $\Omega$ to obtain  
\begin{equation*}
	\label{electrostatic_prope4}
	0 = \int_{\Omega} (1+v(t))\partial_t \phi (t) \mathcal{L}_{v(t)} \phi(t)\, \rd(x,\eta)
\end{equation*}
for $t \in [0,T] $. Using the divergence form of $\mathcal{L}_{v(t)}$,
Gauss' theorem, and the fact that $\partial_t \phi (t) =0$ on $\partial \Omega$, we get 
\begin{align*}
	\label{electrostatic_prope5}
	0 &= - \int_{\Omega} \nabla\big( (1+v(t))  \partial_t \phi(t)\big) \cdot 
	(\bm{\alpha}(t)\nabla \phi(t))  \,\rd(x,\eta) 
	\nonumber \\%[0.1cm]
	& \quad+ \int_{\Omega} (1+v(t))\partial_t \phi (t) 
	(\bm{b}(t) \cdot \nabla \phi(t))\, \rd(x,\eta), \quad t \in [0,T].
\end{align*}
We further deduce from the definitions of $\alpha_i(t)$, $b_i(t)$, and $V(t)$ that 
\begin{align*}
	0 &= - \varepsilon^2 \int_{\Omega} \big(\partial_t \phi (t)  \nabla v(t) + 
	(1+v(t)) \nabla'\partial_t \phi (t)\big) \cdot \big(\nabla' \phi(t) 
	- \eta V(t) \partial_{\eta} \phi (t) \big) \, \rd(x,\eta)
	\\%[0.1cm]
	& \quad  + \varepsilon^2 \int_{\Omega} \eta  (1+v(t)) V(t)
	\partial_{\eta}\partial_t \phi (t) \cdot \big( \nabla' \phi(t) 
	- \eta V(t)  \partial_{\eta} \phi (t) \big) \, \rd(x,\eta)
	\\%[0.1cm]
	& \quad - \int_{\Omega} \frac{\partial_{\eta}\partial_t \phi(t) \partial_{\eta} \phi(t)}
	{1+v(t)}\, \rd(x,\eta) 
	\\%[0.1cm]
	& \quad + \varepsilon^2 \int_{\Omega} \partial_t \phi (t) \nabla v(t)
	\cdot \big( \nabla' \phi(t) - \eta V(t) \partial_{\eta} \phi (t)  \big)\, 
	\rd(x,\eta)
\end{align*}
and rearranging gives 
\begin{align*}
	0 & = - \varepsilon^2 \int_{\Omega} (1+v(t)) \big( \nabla' \phi(t) - \eta V(t)
	\partial_{\eta} \phi (t) \big) \cdot \big( \nabla' \partial_t \phi(t) 
	- \eta V(t) \partial_{\eta} \partial_t \phi (t) \big)\, \rd(x,\eta)
	\\%[0.1cm]
	& \quad - \int_{\Omega} \cfrac{\partial_{\eta}\partial_t \phi(t) \partial_{\eta} \phi(t)}
	{1+v(t)}\, \rd(x,\eta) .
\end{align*}
Combining this with \eqref{electrostatic_prope3} yields
\begin{align*}
	\frac{\rd}{\rd t} E_e(v(t))
	&= -2 \varepsilon^2 \int_{\Omega} \big( \nabla' \phi(t) 
	- \eta V(t) \partial_{\eta} \phi (t) \big) \cdot \partial_t V(t) 
	\partial_{\eta} \phi (t) \eta (1+v(t))\, \rd(x,\eta)
	\\%[0.1cm]
	& \quad + \varepsilon^2 \int_{\Omega} \vert \nabla' \phi(t) 
	- \eta V(t) \partial_{\eta} \phi (t) \vert^2  \, \partial_t v(t) \, \rd(x,\eta) 
	\\%[0.1cm]
	&\quad 
	- \int_{\Omega} ( \partial_{\eta} \phi (t))^2 \frac{\partial_t v(t)}
	{(1+v(t))^2}\, \rd(x,\eta) 
\end{align*}
for $t \in [0,T]$.  Returning back to the original variables $(x,z)$ and potential $\psi(t)$, and using the 
identity 
\begin{equation*}
	\partial_t V(t) 
	= \partial_t \Big(\nabla \ln{(1+v(t))}\Big)
	= \nabla \left( \frac{\partial_t v(t)}{1+v(t)} \right), \quad t \in [0,T],
\end{equation*}
we get, from Green's formula and from the fact that $\partial_t v(t)=0$ on $\partial D$, 
\begin{align*}
	\label{electrostatic_prope7}
	\frac{\rd}{\rd t} E_e(v(t))
	&= -2 \varepsilon^2 \int_{\Omega(v(t))} (1+z) \partial_z \psi (t) \nabla' 
	\psi (t)\cdot \partial_t V(t)\, \rd(x,z)
	\nonumber \\%[0.1cm]
	& \quad +  \int_{\Omega(v(t))} \Bigl( \varepsilon^2 
	\vert \nabla' \psi(t) \vert^2- (\partial_z \psi (t))^2 
	\Bigr)  \frac{\partial_t v(t)}{1+v(t)}\, \rd(x,z)
	\nonumber \\%[0.1cm]
	&= 2 \varepsilon^2 \int_{\Omega(v(t))} (1+z) 
	\Big( \partial_z \psi(t) \Delta' \psi(t) + 
	\nabla' \psi(t) \cdot \nabla'\partial_z \psi(t)\Big)
	 \frac{\partial_t v(t)}{1+v(t)}\, \rd(x,z)
	 \nonumber \\%[0.1cm]
	 & \quad + 2 \varepsilon^2 \int_D 
	 \partial_z \psi(t,\cdot,v(t))  \nabla'\psi(t,\cdot,v(t)) \cdot
	 \nabla v(t) \partial_t v(t) \, \rd x
	 	\nonumber \\%[0.1cm]
	 & \quad +  \int_{\Omega(v(t))} \Big( \varepsilon^2 
	 \vert \nabla' \psi(t) \vert^2- (\partial_z \psi (t))^2 
	 \Big)  \frac{\partial_t v(t)}{1+v(t)}\, \rd(x,z)
,\quad t \in [0,T].
\end{align*}
We next use the Laplace equation $\varepsilon^2 \Delta' \psi(t) +\partial_z^2 \psi(t) = 0$
and again Green's formula to find 
\begin{align*}
	%\label{electrostatic_prope8}
	\frac{\rd}{\rd t} E_e(v(t))
	&= \int_{\Omega(v(t))} (1+z) \Big( - \partial_z\big( (\partial_z \psi(t))^2 \big) 
	+ \varepsilon^2 \,\partial_z \big( \vert \nabla' \psi(t)\vert^2 \big) 
	\Big) \frac{\partial_t v(t)}{1+v(t)}\,  \rd(x,z)
	\nonumber \\
	& \quad + 2  \varepsilon^2 \int_D 
	\partial_z \psi(t,\cdot,v(t))  \nabla'\psi(t,\cdot,v(t)) \cdot
	\nabla v(t) \partial_t v(t)  \, \rd x
	\nonumber\\
	& \quad +   \int_{\Omega(v(t))} \Bigl( \varepsilon^2 
	\vert \nabla' \psi(t) \vert^2-  (\partial_z \psi (t))^2 
	\Bigr)  \frac{\partial_t v(t)}{1+v(t)}\, \rd(x,z)
	\nonumber\\
	& \quad  = \int_D \Big( \varepsilon^2 \vert \nabla' \psi(t,\cdot,
	v(t))\vert^2- ( \partial_z \psi(t,\cdot,v(t)))^2 \Big)
	\partial_t v(t) \, \rd x
	\nonumber \\
	& \quad + 2  \varepsilon^2 \int_D  
	\partial_z \psi(t,\cdot,v(t))  \nabla'\psi(t,\cdot,v(t)) \cdot
	\nabla v(t)\partial_t v(t) \, \rd x ,\quad t \in [0,T].
\end{align*}
Using the identity 
\begin{equation}
	\label{kidentity}
	\nabla' \psi(t,x,v(t))=-\partial_{z}\psi(t,x,v(t))  \nabla v(t),
	\quad (t,x)\in [0,T]\times D,
\end{equation}
which follows from differentiating the boundary condition $\psi(t,x,v(t))=1$, we infer 
\begin{equation*}
\frac{\rd}{\rd t} E_e(v(t)) = - \int_D ( 1 + \varepsilon^2 
	\vert \nabla v(t) \vert^2) (\partial_z \psi(t,\cdot,
	v(t)))^2 \partial_t v(t)\, \rd x,\quad t \in  [0,T].
\end{equation*}
Integrating this over $[t_1,t_2]$ we obtain
\begin{equation*}
	E_e(v(t_2)) - E_e(v(t_1)) 
	=- \int_{t_1}^{t_2} \int_D g(v(s))\partial_t v(s)\, \rd x \rd s, \quad 0\leq t_1\leq t_2 \leq T, 
\end{equation*}
and the proposition is proved.
\end{proof}

Note that the proof of Proposition \ref{electrostaticE} uses only the first boundary condition of $W^{4\xi}_{2,\mathcal{B}}(D)$. 

\begin{remark}
An alternative approach to compute the derivative ~$\frac{\rd}{\rd t} E_e(v(t))$ and to investigate differentiability properties of $E_e$ is presented in \cite[Section 4]{LaurencotWalker2019}. This approach is based on a transformation that maps $\Omega(w(t))$ onto $\Omega(v(t))$ instead of the transformation $T_{v(t)}$ to a fixed cylinder. 
	%Doing this transformation allows one 
	%to rewrite  $E_e(w(t))$, for each $w(t)$ in a neighborhood of $v(t)$, as an integral over 
	%$\Omega(v(t))$ and then to study the behavior of $E_e(w(t))-E_e(v(t))$ as $w(t) \rightarrow v(t)$.
\end{remark}

%%%%%%%%%%%%%%%%
%%%%%%%%%%%%%%%%
\section{Well-Posedness: Proof of Theorem \ref{wellposedness}}
\label{wp}

%%%%%%%%%%%%%%%%
%%%%%%%%%%%%%%%%

Thanks to Proposition \ref{ellipticp} we may rewrite \eqref{evolution1}-\eqref{evolution5} as a semilinear evolution equation 
\begin{equation}
\label{CauchyProblem}
\left. \begin{array}{rl}
\partial_t u + A u \hspace{-0.2cm}&= -\lambda  g(u), \quad  t>0,
		\\[0.1cm]
u(0) \hspace{-0.2cm}&=u^0,
	\end{array}\right.
\end{equation}
only involving $u$, where the operator $A\in \mathcal{L}(W^4_{2,\mathcal{B}}(D),L_2(D))$ is 
defined by 
\begin{equation*}
	Av:=(\beta \Delta^2 - \tau \Delta)v, \quad v
	\in \text{dom}(A)=W^4_{2,\mathcal{B}}(D) .
\end{equation*}
Note that the hinged boundary conditions \eqref{evolution2} are incorporated in the domain of $A$. 
Once we solve \eqref{CauchyProblem} for $u$ we find the solution $\psi_{u(t)}$ to 
\eqref{evolution4}-\eqref{evolution5} via Proposition \ref{ellipticp}. 

First, we show the following property of $A$. 

\begin{lemma}
	\label{aSemigroup}
	It holds that 
	\begin{equation*}
		A \in \mathcal{H}(W^4_{2,\mathcal{B}}(D),L_2(D)),
	\end{equation*}
	i.e. $-A$ is the generator of a strongly continuous analytic semigroup  ~$\{e^{-tA} \, : \, t\geq 0\}$ on $L_2(D)$.
\end{lemma}

\begin{proof}
We shall apply \cite[Remark 4.2(b)]{Amann1993} to prove the result. 
Setting  $D_j:=-i\partial_{x_j}$ for $j=1,2$, we observe that the operator $A$ may be written as 
\begin{equation*}
		Av= \beta \sum_{k,l=1}^{2} D_k^2 D_l^2 v + \tau \sum_{k=1}^{2} D_k^2 v .
\end{equation*}
It is easily seen that the principal symbol associated to $A$ is given by 
\begin{equation}
\label{prsymbolA}
	a_{\pi}(\xi):=\beta \vert \xi \vert^4, \quad \xi \in \mathbb{R}^2 ,
\end{equation}
and since $\beta >0$, the spectrum  $\sigma(a_{\pi}(\xi))$ satisfies 
\begin{equation*}
%\label{ssymbolA}
	\sigma(a_{\pi}(\xi)) \subset \{ z \in \mathbb{C}\, :  \,\text{Re } z > 0 \} \; \text{ for all } \,  \xi \in \mathbb{S}_1,
\end{equation*}
with $\mathbb{S}_1$ denoting the unit sphere in $\mathbb{R}^2$. This means that $A$ is normally elliptic (see \cite[p.18]{Amann1993} for a definition). 

It follows from \eqref{evolution2} that the system $\mathcal{B}$ of boundary operators is $\mathcal{B} := (\mathcal{B}_1, \mathcal{B}_2)$, where 
\begin{equation}
	\label{boundaryOperator}
	\mathcal{B}_1 v = \text{tr } v, \quad 
	\mathcal{B}_{2}v =- \sum_{k=1}^2 \text{tr } D_k^2 v - 
	(1 -\sigma) i \kappa \sum_{k=1}^2 \nu_k  \text{tr } D_k v
\end{equation}
for $v \in W^4_{2,\mathcal{B}}(D)$, with $\text{tr}$ denoting the trace operator on $\partial D$. 
Since ~$\partial D \in C^{4}$, we have ~$\nu \in C^{3}(\partial D)$ and ~$\kappa \in C^{2}(\partial D)$, and hence
\begin{equation}
%\label{coefboundaryOperator}
	(1 - \sigma) i \kappa \nu_k \in C^{2}(\partial D, \mathbb{C}), \quad k= 1,2 .
\end{equation}
The principal boundary symbol of $\mathcal{B}$ is given by
\begin{equation*}
	%\label{prsymbolB}
	b_{\pi} (\xi) := (1 ,- \vert \xi \vert^2), \quad \xi \in 
	\mathbb{R}^2 .
\end{equation*}
Recall from \cite[p.18]{Amann1993} that $(A, \mathcal{B})$ is said to be normally elliptic if $A$ is normally elliptic and $\mathcal{B}$ satisfies the normal complementing condition with respect to $A$. The latter condition is also called the Lopatinskii-Shapiro condition and 
requires that, for any $x \in \partial D$, $\xi \in  \mathbb{R}^2$ with $\xi \cdot \nu(x) = 0$ and any $\mu\in\mathbb{C}$ with $\text{Re }\mu \geq 0$ and $\vert \mu \vert + \vert \xi \vert  \neq 0$, 
zero is the only exponentially decaying solution of the initial value
problem on $[0,\infty)$: 
\begin{empheq}{align*}
\big( \mu + a_{\pi}(\xi + \nu(x) i \partial_{t}) \big) v &=0, 
	\\%[-0.05cm] 
b_{\pi}(\xi + \nu(x) i \partial_{t})v(0)  &=0.
\end{empheq}
By \cite[Remark 4.2(b)]{Amann1993} we obtain $A \in \mathcal{H}(W^4_{2,\mathcal{B}}(D),L_2(D))$ provided $(A,\mathcal{B})$ is normally elliptic. So we need only to verify the normal complementing condition. First observe that, with the definitions of $a_{\pi}$ and $b_{\pi}$, the above initial value problem is written as
\begin{empheq}{align}
\big( \mu + \beta (\vert \xi \vert^2 - \partial_t^2)^2 \big) v(t) &=0, \quad t>0, 
\label{Lopa.1}   
\\
 v(0) = \partial_t^2 v(0) &=0.
\label{Lopa.2}
\end{empheq}
If $\mu =0$, then the general solution to \eqref{Lopa.1} is 
\begin{equation*}
	v(t)= (C_1 + C_2 t)e^{\vert \xi \vert t} + (C_3 + C_4 t)
	e^{-\vert \xi \vert t}, \qquad t \geq 0,
\end{equation*}
with $C_k \in \mathbb{R}$ for $k=1,\dots, 4$. In this case $\xi \neq 0$. Since the solution $v$ must decay exponentially, we must have $C_1=C_2=0$. Imposing that $v$ satisfies also the initial conditions \eqref{Lopa.2}, we get $C_3=C_4=0$ and therefore $v \equiv 0$. 

When $\mu \neq 0$, the characteristic equation of \eqref{Lopa.1} is  
\begin{equation*}
	r^4 - 2 \vert \xi \vert^2 r^2+ \vert \xi\vert^4 + \frac{\mu}{\beta} =0,
\end{equation*}
and its roots are 
\begin{equation*}
	r_k = \pm \sqrt{\vert \xi \vert^2 \pm i \frac{\sqrt{\mu}}{\sqrt{\beta}}} \, ,
	\quad k=1,\dots,4
\end{equation*}
(recall that the square root of a complex number has a nonnegative real part). Writing $\sqrt{\mu}=a_{\mu}+i b_{\mu}$ with $a_{\mu}, b_{\mu} \in \mathbb{R}$, $a_{\mu} \geq 0$ 
we have that $\text{Re}\,\mu \geq 0$ implies $a_{\mu} > 0$. Hence, 
\begin{equation*}
	r_{1,3} = \pm \sqrt{ \bigg( \vert \xi \vert^2 -\frac{b_{\mu}}
		{\sqrt{\beta}}\bigg) + i \frac{a_{\mu}}{\sqrt{\beta}}} \,, 
	\quad 
	r_{2,4} = \pm \sqrt{ \bigg( \vert \xi \vert^2 + \frac{b_{\mu}}
		{\sqrt{\beta}}\bigg) - i  \frac{a_{\mu}}{\sqrt{\beta}}} \, ,
\end{equation*}
with $r_1$, $r_2$ having positive real part. Therewith, since each root has multiplicity one, the general solution of \eqref{Lopa.1} is given by
\begin{equation*}
	v(t)=\sum_{k=1}^4 C_k e^{r_k t},\quad t\geq 0,
\end{equation*}
with $C_k \in \mathbb{R}$ for $k=1,\dots, 4$, and, since $v$ must decay exponentially, $C_1=C_2=0$. Invoking \eqref{Lopa.2}, we find that 
\begin{equation*}
	0=v(0) = C_3 + C_4, \quad 0 = \partial_t^2 v(0) = C_3 r_3^2 + C_4 r_4^2.
\end{equation*}
The initial value problem \eqref{Lopa.1}-\eqref{Lopa.2} has for 
$(\xi,\mu) \in \mathbb{R}^2 \times \{z \in \mathbb{C}\, : \,\text{Re }z \geq 0 \}$ with $\vert \mu \vert + \vert \xi \vert \neq 0$ only the trivial solution $v \equiv 0$ if 
\begin{equation}
	\label{Lopa.3}
	r_4^2 - r_3^2 = (r_4 + r_3)(r_4 - r_3) \neq 0.
\end{equation}
Because $r_4 \neq  \pm r_3$, condition \eqref{Lopa.3} holds. Consequently, $(A,\mathcal{B})$ is normally elliptic.
\end{proof}

Next, we show that the spectrum of $-A$ lies in the left half-plane $\{z \in \mathbb{C}\,:\,\text{Re }z <0\}$.

\begin{lemma}
	\label{spectrumSG}
	We have 
	\begin{equation*}
		\sigma(-A) \subset \{ z \in \mathbb{C}\, : \, \emph{Re }z <0 \}.
	\end{equation*}
\end{lemma}

\begin{proof}
Due to the compact embedding of $W^4_{2,\mathcal{B}}(D)$ in $L_2(D)$, and since ~$A \in \mathcal{H}(W^4_{2,\mathcal{B}}(D),L_2(D))$, the operator $-A$ has compact resolvent. 
In view of \cite[Theorem 6.29]{Kato1995}, its spectrum consists only of isolated eigenvalues of finite 
multiplicity. If $\mu\in \mathbb{C}$ is such an eigenvalue of $-A$ with a corresponding eigenfunction $\varphi \in W^4_{2,\mathcal{B}}(D,	\mathbb{C})$, then testing the equation 
\begin{equation*}
		(-\beta \Delta^2 + \tau \Delta) \varphi = \mu \varphi
\end{equation*}
by its complex conjugate $\overline{\varphi}(x) = \overline{\varphi (x)}$ yields 
\begin{equation*}
\mu \int_D \vert \varphi \vert^2 \, \rd x = \int_D  \overline{\varphi} 
( -\beta  \Delta^2 \varphi + \tau  \Delta \varphi ) \, \rd x. 
\end{equation*}
Writing $\varphi: \overline{D} \rightarrow \mathbb{C}$ in terms of 
its real and imaginary parts, say ~$\varphi(x) = \varphi_1(x) + i \varphi_2(x)$ for ~$\varphi_1(x), 
\varphi_2(x) \in \mathbb{R}$, and noticing that 
\begin{equation}
	\label{complexcon.1}
	\vert \nabla \varphi \vert^2 = \sum_{j=1}^2 \vert \nabla \varphi_j \vert ^2, \quad   
	\vert \partial_{\nu} \varphi \vert^2 = \sum_{j=1}^2 (\partial_{\nu} \varphi_j)^2, 
	\quad
	\vert \Delta \varphi \vert^2 = \sum_{j=1}^2  ( \Delta \varphi_j )^2, 
\end{equation}
we obtain, using Green's formula twice and the boundary conditions for $\varphi$, that
\begin{align*}
	-\beta \int_D \overline{\varphi} \, \Delta^2 \varphi \, \rd x 
	&= \beta \int_D \nabla \Delta \varphi \cdot \nabla \overline{\varphi}\, \rd x
= \beta \int_{\partial D} \Delta \varphi \, \partial_{\nu} \overline{\varphi}\, \rd\omega 
- \beta \int_D  \Delta \varphi \, \Delta \overline{\varphi}\, \rd x
\\
&= \beta (1-\sigma) \int_{\partial D} \kappa \vert \partial_{\nu} \varphi \vert^2 \, \rd\omega 
-\beta \int_D \vert \Delta \varphi \vert^2 \, \rd x.
\end{align*}
Again by applying Green's formula and the fact that $\varphi=0$ on $\partial D$ we get 
\begin{equation*}
\tau \int_D \overline{\varphi} \, \Delta \varphi \, \rd x= - \tau  \int_D \vert \nabla \varphi \vert^2 \, \rd x ,
\end{equation*}
whence
\begin{equation*}
\mu \int_D \vert \varphi \vert^2 \, \rd x 
= \beta (1-\sigma) \int_{\partial D} \kappa  \vert \partial_{\nu} \varphi \vert^2 \, \rd\omega 
- \beta \int_D \vert \Delta \varphi\vert^2 \, \rd x - \tau \int_D \vert \nabla \varphi \vert^2 \, \rd x. 
\end{equation*}
By the identity established in \cite[Lemma A.1]{SweersVassi2018} we find that 
\begin{align*}
\beta (1-\sigma) \int_{\partial D} &\kappa ( \partial_{\nu} \varphi_j )^2 \, \rd\omega  
= -2 \beta (1-\sigma) \int_D \big( (\partial_{x_2}\partial_{x_1} \varphi_j )^2 - 
\partial_{x_1}^2 \varphi_j  \, \partial_{x_2}^2 \varphi_j  \big)\, \rd x 
\nonumber \\
&= -  \beta (1-\sigma) \int_D \big( (\partial_{x_1}^2 \varphi_j )^2 
+ (\partial_{x_2}^2 \varphi_j )^2 + 2 (\partial_{x_2}\partial_{x_1} \varphi_j )^2 - ( \Delta \varphi_j 
	)^2 \big)\, \rd x
%\label{complexcon.2}
\end{align*}
for $j=1,2$ and hence, together with \eqref{complexcon.1}, we infer that 
\begin{align*}
\mu \int_D \vert \varphi \vert^2 \, \rd x  
&= - \beta (1-\sigma) \int_D \big( \vert \partial_{x_1}^2 \varphi\vert^2 
+ \vert \partial_{x_2}^2 \varphi \vert^2 
+ 2 \vert \partial_{x_2} \partial_{x_1} \varphi \vert^2 \big)\, \rd x 
\\
& \quad
- \beta \sigma \int_D \vert \Delta \varphi \vert^2 \, \rd x 
- \tau \int_D \vert \nabla \varphi \vert^2 \, \rd x.
\end{align*}
Now using Young's inequality gives 
\begin{equation*}
%\label{complexcon.4}
\frac{1}{2} (\Delta \varphi_j)^2 \leq (\partial_{x_1}^2 \varphi_j)^2 + (\partial_{x_2}^2 \varphi_j)^2 + 2 (\partial_{x_2}\partial_{x_1} \varphi_j)^2 \, \text{ in } \,  D
\end{equation*}
for $j=1,2$, so that 
\begin{align}
\mu \int_D \vert \varphi \vert^2 \, \rd x &\leq -\frac{1}{2} \beta (1-\sigma) \int_D \vert 
\Delta \varphi \vert^2 \, \rd x - \beta \sigma\int_D \vert \Delta \varphi \vert^2 \, \rd x
-\tau \int_D \vert \nabla \varphi \vert^2 \, \rd x
	\nonumber \\
& = - \frac{1}{2} \beta (1+\sigma) \int_D \vert \Delta \varphi \vert^2 \, \rd x
-\tau \int_D \vert \nabla \varphi \vert^2 \, \rd x.
\label{complexcon.6}
\end{align}
Consequently, since $\beta (1+ \sigma) >0$ and $\tau \geq 0$, 
\begin{equation*}
\mu \int_D \vert \varphi \vert^2 \, \rd x  \leq 0
\end{equation*}
and thus $\mu \leq 0$.

 It remains to show that $\mu<0$. If $\mu=0$, then, due to  \eqref{complexcon.6},
 \begin{equation*}
 - \frac{1}{2} \beta (1+\sigma) \int_D \vert \Delta \varphi \vert^2 \, \rd x 
 -\tau \int_D \vert \nabla \varphi \vert^2 \, \rd x =0,
 \end{equation*}
and since $\tau \geq 0$, we get 
\begin{equation*}
%\label{spectrumSG.2}
- \frac{1}{2} \beta (1+\sigma) \int_D \vert \Delta \varphi \vert^2 \, \rd x =0.  
\end{equation*}
So $\Delta \varphi =0$ in $D$. Since ~$\varphi = 0$ on $\partial D$, it follows that ~$\varphi \equiv 0$, which is a contradiction. Thus ~$\mu<0$. 
\end{proof}

Thanks to the above lemma we can use \cite[Theorem 4.4.3]{Pazy1983} to infer that the 
semigroup ~$\{e^{-tA} \, : \, t\geq 0\}$ has exponential decay, i.e. there are $M \geq 1$ and $\varpi >0$ such that 
\begin{equation*}
	\Vert e^{-tA} \Vert_{\mathcal{L}(L_2(D))} \leq M e^{-\varpi t}, \quad t\geq 0.
\end{equation*}
Moreover, the semigroup has the following regularizing properties, which we will use later. 

\begin{lemma}
\label{regularizingsemigroup}
There exists $\varpi>0$ such that the following holds. If $0\leq \gamma \leq \alpha \leq 1$ with ~$4\alpha,  4\gamma  \notin \left\{\frac{1}{2}, \frac{5}{2} \right\}$, then 
\begin{equation*}
\Vert e^{-tA} \Vert_{\mathcal{L}(W^{4\gamma}_{2,\mathcal{B}}(D), 
W^{4\alpha}_{2,\mathcal{B}}(D))} \leq M e^{-\varpi t}t^{\gamma -\alpha}, \quad t>0,
\end{equation*}
for some number $M\geq 1$ depending on $\alpha$ and $\gamma$. 
\end{lemma}

\begin{proof}
We denote by $\left[\cdot, \cdot\right]_{\theta}$ the 
complex and by $\left(\cdot,\cdot\right)_{\theta,q}$, $1 \leq q\leq\infty$, 
the real interpolation functor for $0<\theta<1$. We can easily check that the system of boundary operators $\mathcal{B}:=(\mathcal{B}_1,\mathcal{B}_2)$ given in \eqref{boundaryOperator} forms a normal system in the sense of \cite[Definition 4.3.3(1)]{Triebel1995}. Then by \cite[Theorem 4.3.3(a)]{Triebel1995}, we obtain
\begin{equation}
\label{regularizingsemigroup.1}
\left( L_2(D), W^4_{2,\mathcal{B}}(D)\right)_{\theta} \doteq W^{4\theta}_{2,\mathcal{B}}(D) \, \text{ if } \, 4\theta \notin \left\{\frac{1}{2}, \frac{5}{2} \right\},
\end{equation}
where
\begin{equation*}
\left(\cdot,\cdot\right)_{\theta} :=\left\{
\begin{array}{ll}
\left(\cdot,\cdot\right)_{\theta,2} &\text{ if } \, 4\theta \notin \left\{ 1,2,3\right\} ,
\\[0.1cm]
\left[ \cdot, \cdot \right]_{\theta} &\text{ if } \, 4\theta \in \left\{ 1,2,3\right\}.
\end{array} 
\right.
\end{equation*}
Next let $E_0:= L_2(D)$, $E_1:=W^4_{2,\mathcal{B}}(D)$, and set 
\begin{equation*}
E_{\theta} := \left( L_2(D), W^4_{2,\mathcal{B}}(D)\right)_{\theta} \doteq W^{4\theta}_{2,\mathcal{B}}(D), \quad 4\theta \notin \left\{\frac{1}{2}, \frac{5}{2} \right\}. 
\end{equation*}
In view of Lemmas \ref{aSemigroup} and \ref{spectrumSG}, we can 
apply \cite[Theorem V.2.1.3]{Amann1995} to conclude that there 
are $\varpi >0$ and $M \geq 1$ such that 
\begin{equation*}
\Vert e^{-tA} \Vert_{\mathcal{L}(E_{\gamma},E_{\alpha})} \leq M e^{-\varpi t} t^{\gamma-\alpha},\quad t>0,
\end{equation*}
where $0\leq \gamma \leq \alpha \leq 1$ with $4\alpha, \, 4\gamma \notin \left\{\frac{1}{2}, \frac{5}{2} \right\}$, and $M$ depends on $\alpha$ and $\gamma$. 
\end{proof}

We are now in a position to prove the well-posedness of \eqref{CauchyProblem}. \\
\\
\noindent{\textit{ Proof of Theorem \ref{wellposedness}.}}
Let $4\xi \in (\frac{7}{3},4)\backslash \{\frac{5}{2}\}$ be fixed and consider an initial value ~$u^0 \in W^{4\xi}_{2,\mathcal{B}}(D)$ such that $u^0 >-1$ in $D$. Due to the continuous 
embedding of $W^{4\xi}_2(D)$ in $W^2_3(D)$ and in $C(\overline{D})$, there are $c_W>1$ and 
$\rho \in (0,\tfrac{1}{2})$ such that $\Vert v\Vert_{W^2_3(D)} \leq c_W  \Vert v\Vert_{W^{4\xi}_2(D)}$ 
for all $v\in W^{4\xi}_2(D)$ and 
\begin{equation*}
%\label{S3property}
	u^0 \in S_3(2\rho), \quad  \Vert u^0 \Vert_{W^{4\xi}_{2, \mathcal{B}}(D)} \leq \frac{1}{2 \rho}.
\end{equation*}
According to Lemma \ref{regularizingsemigroup}, it holds that
\begin{equation}
\label{regu}
\Vert e^{-tA} \Vert_{\mathcal{L}(W^{4\xi}_{2,\mathcal{B}}(D))} + t^{\xi} 
\Vert e^{-tA} \Vert_{\mathcal{L}(L_2(D), W^{4\xi}_{2,\mathcal{B}}(D))} \leq M e^{-\varpi t}, \quad t\geq0. 
\end{equation}
Next, set $\rho_0:= \frac{ \rho }{M c_1} \in (0,\rho)$ and recall from Proposition \ref{ellipticp} that, for 
$v_1,v_2 \in \overline{S}_3(\rho_0)$, 
\begin{equation}
	\Vert g(v_1) - g(v_2)\Vert_{L_2(D)} \leq c_l \Vert v_1 - v_2 \Vert_{W^2_3(D)},
	\label{lipg}
\end{equation}
and that 
\begin{equation}
	\Vert g(v) \Vert_{L_2(D)} \leq c_b \;\text{ for all } \; v \in \overline{S}_3(\rho_0)
	\label{boundg}
\end{equation}
with constants $c_l, c_b>0$ depending only on $\rho$, $\varepsilon$, and $D$. Given $T>0$ we introduce the complete metric space 
\begin{equation*}
	\mathcal{V}_T := C([0,T],\overline{S}_3(\rho_0))
\end{equation*}
and define 
\begin{equation*}
	\Lambda(v)(t) := e^{-tA} u^0 - \lambda \int_0^t e^{-(t-s)A} g(v(s)) \, \rd s
%\label{VoC}
\end{equation*}
for $t\in [0,T]$ and $v\in \mathcal{V}_T$. Arguing similar to the proof of \cite[Proposition 4.3]{LaurencotWalkerII} with the aid of \cite[Theorems II.5.2.1 \& II.5.3.1]{Amann1995}, we infer from 
\eqref{regu}, \eqref{lipg}, and \eqref{boundg} that the mapping $\Lambda: \mathcal{V}_T \rightarrow \mathcal{V}_T$ defines a contraction for each $\lambda>0$ provided that $T:= T(\rho, \varepsilon,\lambda)>0$ is sufficiently small. Hence $\Lambda$ has a unique fixed point $u$ in $ \mathcal{V}_T$ which solves \eqref{CauchyProblem} with the regularity specified in \eqref{regularity.1}. This readily gives parts (i)-(ii) of Theorem \ref{wellposedness} and it remains to show the global existence statement (iii) therein. To this end, proceeding similar to the proof of \cite[Theorem 2]{EscherLaurencotWalker2014}, we find that there are $\lambda_*>0$ and $N>0$ such that $\Lambda: \mathcal{V}_T \rightarrow \mathcal{V}_T$ is a contraction for each $T>0$ provided that $\lambda \in (0,\lambda_*)$ and ~$\Vert u^0 \Vert _{W^{4\xi}_{2, \mathcal{B}}(D)} \leq N$. Hence $u$ is a global solution to \eqref{CauchyProblem}, and Theorem \ref{wellposedness} (iii) follows. 
\qed
\medskip

We shall prove the refined criterion for global existence stated in Theorem ~\ref{globalcriterion} in Section ~\ref{rc}, the starting point being Theorem  ~\ref{wellposedness} (ii). 

%%%%%%%%%%%%%%%%
%%%%%%%%%%%%%%%%

%%%%%%%%%%%%%%%%
%%%%%%%%%%%%%%%%
\section{Energy Equality: Proof of Theorem \ref{Tenergyequality}}
\label{ee}
%%%%%%%%%%%%%%%%
%%%%%%%%%%%%%%%%
The goal of this section is to prove Theorem  \ref{Tenergyequality}. For this 
we shall use Proposition ~\ref{electrostaticE}. 
Under the assumptions of Theorem \ref{wellposedness} let $(u,\psi_u)$ be the solution 
to \eqref{evolution1}-\eqref{evolution5}. Unfortunately, we cannot directly apply Proposition 
 \ref{electrostaticE} as $u$ only belongs to $C^1((0,T_{m}),L_2(D))$. 
 To get around this difficulty we use an approximation argument, which follows the idea of \cite[Section 4]{LaurencotWalkerI}. Let 
 \begin{equation*}
 	u_{\delta}(t,x) := \frac{1}{\delta} \int_t^{t+\delta} u(s,x)\, \rd s, \quad t\in [0,T_{m}), \;
 	x\in D, \; \delta \in (0,T_{m}-t),
 \end{equation*}
be the Steklov average of $u$. Fix $T\in (0,T_{m})$ and let $\delta \in (0,T_{m}-T)$ in the sequel.  Since $u$ belongs to $C([0,T+\delta],W^{4\xi}_{2,\mathcal{B}}(D))$, we get by the fundamental theorem of calculus 
\begin{equation}
	\label{ud_prope1}
	u_{\delta}\in C^1([0,T],W^{4\xi}_{2,\mathcal{B}}(D)) \; \text{ with }\;
	\partial_t  u_{\delta}(t)= \frac{u(t+\delta)-u(t)}{\delta}, \; t\in [0,T].
\end{equation}
In addition, for $t\in [0,T]$, it holds that 
\begin{align*}
	\Vert u_{\delta}(t) - u(t)\Vert_{W^{4\xi}_{2,\mathcal{B}}(D)} &= 
	\left\Vert \frac{1}{\delta} \int_t^{t+\delta} (u(s)-u(t))\, \rd s 
	\right\Vert_{W^{4\xi}_{2,\mathcal{B}}(D)} 
	\\
	&\leq \max_{s\in [t,t+\delta]} \Vert u(s)-u(t)\Vert_{W^{4\xi}_{2,\mathcal{B}}(D)} 
	\longrightarrow 0 \;  \text{ as } \;  \delta \rightarrow 0 
\end{align*}
and thus 
\begin{equation}
	\label{ud_prope2}
	u_{\delta} \rightarrow u \; \text { in } \; C([0,T],W^{4\xi}_{2,\mathcal{B}}(D)) \;
	\text{ as } \; \delta \rightarrow 0.
\end{equation}
Note that, for any $ t_0 \in (0,T)$, $u$ belongs to $C([t_0,T+\delta], 
W^4_{2,\mathcal{B}}(D))$. Then, the estimate
\begin{equation*}
	\Vert u_{\delta}(t) - u(t)\Vert_{W^4_{2,\mathcal{B}}(D)} \leq 
	\max_{s\in [t,t+\delta]} \Vert u(s)-u(t)\Vert_{W^{4}_{2,\mathcal{B}}(D)} 
\end{equation*}
proves that, as $\delta \rightarrow 0$, $u_{\delta}(t)$ converges to $u(t)$ 
in $W^{4}_{2,\mathit{B}}(D)$ uniformly on $[t_0,T]$ for any $t_0\in (0,T)$. 
Since $t_0$ is arbitrary, 
\begin{equation}
	\label{udloc_prope1}
	u_{\delta} \rightarrow u \; \text{ in } \; C((0,T],W^{4}_{2,\mathcal{B}}(D)) \; 
	\text{ as } \; \delta \rightarrow 0.
\end{equation}
Since $u$ belongs to $C^1([t_0,T+\delta],L_2(D))$ for every $t_0 \in (0,T)$, it follows from 
\eqref{ud_prope1} that 
\begin{equation*}
	\partial_t u_{\delta}(t) = \cfrac{1}{\delta} \int_t^{t+\delta} \partial_t u(s)\, \rd s, \quad
	t\in [t_0,T].
\end{equation*}
This implies 
\begin{equation*}
	\Vert \partial_t u_{\delta}(t) - \partial_t u(t) \Vert_{L_2(D)} \rightarrow 0
	\; \text{ as } \; \delta \rightarrow 0,
\end{equation*}
uniformly in $t \in [t_0,T]$ for each $t_0\in (0,T)$, and since $t_0$ is arbitrary,
\begin{equation}
	\label{udloc_prope2}
	\partial_tu_{\delta} \rightarrow \partial_t u \; \text{ in } \; C((0,T],L_2(D)) \;
	\text{ as } \; \delta \rightarrow 0.
\end{equation}
Next, recalling that the mechanical energy $E_m$ is defined in \eqref{energym}, we obtain by direct calculations that $	E_m(u_{\delta}) \in C^1([0,T])$
with derivative 
\begin{align*}
	&\frac{\rd}{\rd t} E_m(u_{\delta}(t))
	\\
	& = \beta \int_D \Delta u_{\delta}(t) \Delta \partial_t u_{\delta}(t) \, \rd x 
	\\
	&\quad + \beta (1-\sigma) \int_D \Big( 2 \partial_{x_2} \partial_{x_1} u_{\delta}(t) 
	\partial_{x_2} \partial_{x_1} \partial_t u_{\delta}(t)  - \partial_{x_2}^2 u_{\delta}(t) 
	\partial_{x_1}^2 \partial_t u_{\delta}(t) - \partial_{x_1}^2 u_{\delta}(t) 
	\partial_{x_2}^2 \partial_t u_{\delta}(t) \Big)  \rd x
	\\
	&\quad + \tau \int_D \nabla u_{\delta}(t) \cdot \nabla \partial_t 
	u_{\delta}(t) \, \rd x
\end{align*}
for $t\in[0,T]$. Using the useful identity 
\begin{align*}
	& \int_D \Big( 2\partial_{x_2} \partial_{x_1} u_{\delta}(t) 
	\partial_{x_2} \partial_{x_1} \partial_t u_{\delta}(t)  - \partial_{x_2}^2 u_{\delta}(t) 
	\partial_{x_1}^2 \partial_t u_{\delta}(t) - \partial_{x_1}^2 u_{\delta}(t)  
	\partial_{x_2}^2 \partial_t u_{\delta}(t) \Big) \rd x
	\\
	& \quad = - \int_{\partial D} \kappa \, \partial_{\nu} u_{\delta}(t) 
	\partial_{\nu} \partial_t u_{\delta}(t) \, \rd\omega ,
\end{align*}
which follows from \cite[Lemma A.1]{SweersVassi2018}, we further get for $t \in [0,T]$ 
\begin{align*}
	\frac{\rd}{\rd t} E_m(u_{\delta}(t))
	& = \beta \int_D \Delta u_{\delta}(t) \Delta \partial_t u_{\delta}(t) \, \rd x
	+ \tau \int_D \nabla u_{\delta}(t) \cdot \nabla \partial_t 
	u_{\delta}(t) \, \rd x
	\\
	&\quad  -\beta (1-\sigma) \int_{\partial D} \kappa \; \partial_{\nu} u_{\delta}(t) \, 
	\partial_{\nu} \partial_t u_{\delta}(t) \, \rd\omega . 
\end{align*}
Applying Green's formula to the first two terms and using $\partial_t u_{\delta}(t) 
= 0$ on $\partial D$ yields for $t \in (0,T]$ that
\begin{align*}
	\frac{\rd}{\rd t} E_m(u_{\delta}(t))
	& = \beta \int_D \Delta^2 u_{\delta}(t)   \partial_t u_{\delta}(t) \, \rd x 
	- \tau \int_D \Delta u_{\delta}(t) \partial_t u_{\delta}(t) \, \rd x 
	\\
	&\quad  + \beta \int_{\partial D}  \big( \Delta u_{\delta}(t) 
	- (1-\sigma)\kappa \, \partial_{\nu} u_{\delta}(t)\big) 
	\partial_{\nu} \partial_t u_{\delta}(t) \, \rd\omega.
\end{align*}
Therefore, in view of the second boundary condition for $u_{\delta}(t)$ 
(due to $u_{\delta}(t) \in W^{4}_{2,\mathcal{B}}(D)$ for~ $t\in(0,T]$), we obtain 
\begin{equation*}
	\frac{\rd}{\rd t} E_m(u_{\delta}(t))
	= \beta \int_D \Delta^2 u_{\delta}(t)  \partial_t u_{\delta}(t) \, \rd x
	- \tau \int_D \Delta u_{\delta}(t) \partial_t u_{\delta}(t) \, \rd x, \quad t \in (0,T],
\end{equation*}
and integrating this equality on $[t_1, t_2]$ we deduce
\begin{align}
	\label{ud_prope7}
	&E_m(u_{\delta}(t_2))-E_m(u_{\delta}(t_1)) 
	\nonumber \\
	& \qquad = \beta \int_{t_1}^{t_2} \int_D \Delta^2 u_{\delta}(s)  \partial_t u_{\delta}(s)\,  \rd x \rd s
	- \tau \int_{t_1}^{t_2} \int_D \Delta u_{\delta}(s) \partial_t u_{\delta}(s) \, \rd x \rd s 
\end{align}
for $0 < t_1\leq t_2\leq T$. 
We are now concerned with the limit of \eqref{ud_prope7} as $\delta \rightarrow 0$. First, 
by \eqref{udloc_prope1} and \eqref{udloc_prope2} we see that
for any $0<t_1\leq t_2\leq T$, the right-hand side of \eqref{ud_prope7} converges to 
\begin{equation*}
	%\label{udloc_prope5}
	\beta \int_{t_1}^{t_2} \int_D \Delta^2 u(s) \,  \partial_t u(s)
	\, \rd x \rd s
	- \tau \int_{t_1}^{t_2} \int_D \Delta u(s) \, \partial_t 
	u(s) \, \rd x\rd s 
\end{equation*}
as $\delta \rightarrow 0$. Second, from \eqref{ud_prope2} and the 
continuous embedding $W^{4 \xi}_2 (D) \hookrightarrow W^2_2(D)$ it follows that 
\begin{equation*}
	%\label{ud_prope10}
	\left \vert E_m(u_{\delta}(t_k)) - 
	E_m(u(t_k)) \right \vert 
	\longrightarrow 0 \; \text{ as } \; 
	\delta \rightarrow 0,
\end{equation*}
$t_k\in [0,T]$, $k=1,2$. Putting these limits together we get
\begin{equation}
	\label{mechanicalE}
	E_m(u(t_2)) - E_m(u(t_1)) = \int_{t_1}^{t_2} \int_D \big( 
	\beta \Delta^2 u(s) - \tau  \Delta u(s) \big) \partial_t u(s) \, \rd x\rd s %, \: 0 < t_1 \leq t_2 \leq T. 
\end{equation}
for $0 < t_1 \leq t_2 \leq T$. Since $u$ belongs to $C([0,T],W^{4\xi}_{2,\mathcal{B}}(D))$ 
and since $W^{4 \xi}_2 (D) \hookrightarrow W^2_2(D)$, we conclude that 
\begin{equation*}
	E_m(u(t_1)) \longrightarrow E_m(u^0)< \infty \;
	\text{ as } \; t_1 \rightarrow 0.
\end{equation*}
This shows that \eqref{mechanicalE} is valid for $t_1=0$.

Next, we consider the electrostatic energy $E_e$ defined in \eqref{energye}. By Proposition \ref{electrostaticE} it holds that
\begin{equation}
	\label{electrostatic_prope9}
	E_e(u_{\delta}(t_2)) -E_e(u_{\delta}(t_1)) 
	= - \int_{t_1}^{t_2} \int_D 
	g(u_{\delta}(s))\partial_t u_{\delta}(s) \rd x\rd s, \quad 0\leq t_1\leq t_2 \leq T, 
\end{equation}
and we are interested in the limit of \eqref{electrostatic_prope9} as $\delta 
\rightarrow 0$. First, since $u(t)>-1$ in $D$, it follows from \eqref{ud_prope2} that $u(t)$ and $u_{\delta}(t)$ belong to $\overline{S}_3(\rho)$ for some $\rho \in (0,1)$ and for $t\in[0,T]$ and 
$\delta \in (0,\delta_0)$ with $\delta_0>0$ sufficiently small. 
Hence, the Lipschitz property of $g$ stated in Proposition \ref{ellipticp} and the continuous embedding of $W^{4\xi}_2(D)$ in  $W^2_3(D)$ entail that, for $t\in [0,T]$, 
\begin{equation*}
	\Vert g(u_{\delta})(t) - g(u)(t) \Vert_{L_2(D)} 
	\leq c_1 \Vert  u_{\delta}(t) - u(t)\Vert_{W^{4\xi}_{2,\mathcal{B}}(D)}
\end{equation*}
with $c_1=c_1(\rho, D)>0$, whence 
\begin{equation}
	\label{ud_prope3}
	g(u_{\delta}) \rightarrow g(u) \; \text{ in } 
	\; C([0,T],L_2(D)) \; \text{ as } \; \delta \rightarrow 0
\end{equation}
by virtue of \eqref{ud_prope2}. This implies that 
\begin{equation*}
	g(u_{\delta}) \rightarrow g(u) \; \text{ in } 
	\; L_2(0,T;L_2(D)) \; \text{ as } \; \delta \rightarrow 0,
\end{equation*}
and together with \eqref{udloc_prope2} and H\"older's inequality 
we deduce that, for any $t_0\in (0,T)$,
\begin{equation*}
	g(u_{\delta})\partial_t u_{\delta} \rightarrow g(u)\partial_t u  
	\; \text{ in } \; L_1(t_0,T;L_1(D)) \; \text{ as } \; \delta \rightarrow 0.
\end{equation*} 
Thus, for any $0<t_1\leq t_2\leq T$, 
\begin{equation}
	\label{udloc_prope6}
	\left\vert \int_{t_1}^{t_2} \int_D 
	\big( g(u_{\delta}(s)) \partial_t u_{\delta}(s) - g(u(s)) \partial_t u(s) \big)
	\, \rd x\rd s \right\vert \longrightarrow 0 \; \text{ as } \; \delta \rightarrow 0.
\end{equation}
Second, in terms of the coordinates $(x,\eta) \in \Omega$, the electrostatic energy reads
\begin{align}
	\label{ud_prope17} 
	E_e(u_{\delta}(t))
	&= \varepsilon^2 \int_{\Omega} 
	\vert \nabla' \phi_{\delta}(t) - \eta  U_{\delta}(t)  \partial_{\eta} \phi_{\delta}(t) 
	\vert^2 (1+u_{\delta}(t))\, \rd(x,\eta) 
	\nonumber\\
	&\quad + \int_{\Omega} \frac{\left( \partial_{\eta} \phi_{\delta}(t) \right)^2}{1+u_{\delta}(t)}\, \rd(x,\eta), 
\end{align}
where 
\begin{equation*}
	\phi_{\delta}(t) := \phi_{u_{\delta}(t)} \; \text{ and } \; 
	U_{\delta}(t):= \frac{\nabla u_{\delta}(t)}{1 + u_{\delta}(t)}, \; t \in [0,T].
\end{equation*}
We shall show that 
\begin{equation}
	\label{ud_prope16} 
	\vert E_e(u_{\delta}(t_k)) -  E_e(u(t_k))\vert \longrightarrow 0 \; \text{ as } \; \delta \rightarrow 0,
\end{equation}
$t_k \in [0,T]$, $k=1, 2$. From Proposition \ref{ellipticp} we know that, for $t\in [0,T]$,
\begin{equation*}
	\Vert \phi_{\delta}(t) -\phi(t) \Vert_{W^2_2(\Omega)}
	\leq c_0 \Vert u_{\delta}(t)-u(t)\Vert_{W^2_3(D)}
\end{equation*}
with $c_0=c_0(\rho,D)>0$. Using the continuous embedding of $W^{4\xi}_2(D)$ in $W^2_3(D)$ and \eqref{ud_prope2}, we then conclude that 
\begin{equation}
	\label{ud_prope19} 
	\phi_{\delta} \rightarrow  \phi \; \text{ in } \;  
	\quad C([0,T],W^2_2(\Omega)) \;\text{ as } \; \delta \rightarrow 0.
\end{equation}
Again by \eqref{ud_prope2} and the continuous embedding of $W^{4\xi}_2(D)$ in $C^1(\overline{D})$ we obtain 
\begin{equation*}
	U_{\delta} \rightarrow U \; \text{ in } 
	\; C([0,T],L_{\infty}(D)) \; \text{ as } \; \delta \rightarrow 0,
\end{equation*}
which, together with the convergence \eqref{ud_prope19}, yields 
\begin{equation*}
	\nabla'\phi_{\delta} - \eta U_{\delta}  \partial_{\eta} \phi_{\delta} 
	\rightarrow \nabla'\phi - \eta U \partial_{\eta} \phi 
	\;\text{ in } \; C([0,T],L_2(\Omega)) \; \text{ as } \; \delta \rightarrow 0. 
\end{equation*}
Hence, by H\"older's inequality and \eqref{ud_prope2}, we have
\begin{equation*}
	\vert \nabla' \phi_{\delta} - \eta U_{\delta}  \partial_{\eta} \phi_{\delta} 
	\vert^2 (1+u_{\delta})
	\rightarrow \vert \nabla' \phi - \eta U \partial_{\eta} \phi_{\delta} \vert^2 (1+u)
	\;\text{ in } 
	\; C([0,T],L_1(\Omega)) \; \text{ as } \; \delta \rightarrow 0.
\end{equation*}
Also, \eqref{ud_prope2} and \eqref{ud_prope19} imply
\begin{equation*}
	\frac{\left( \partial_{\eta} \phi_{\delta} \right)^2}{1+u_{\delta}}
	\rightarrow \frac{ (\partial_{\eta} \phi )^2}{1+u}
	\; \text{ in } \;  C([0,T],L_1(\Omega)) \; \text{ as } \; \delta \rightarrow 0, 
\end{equation*}
which, together with the previous limit, proves \eqref{ud_prope16}. 
Combining \eqref{electrostatic_prope9}, \eqref{udloc_prope6}, and \eqref{ud_prope16} gives
\begin{equation}
	\label{electrostaticE_2}
	E_e(u(t_2)) - E_e(u(t_1)) = - \int_{t_1}^{t_2} \int_D 
	g(u(s)) \partial_t u(s) \, \rd x\rd s, 
	\quad 0< t_1\leq t_2 \leq T.
\end{equation}
Now since $u$ belongs to $C([0,T],W^{4\xi}_{2,\mathcal{B}}(D))$, we can 
repeat arguments quite similar to those above to conclude that 
\begin{equation*}
	E_e(u(t_1)) \longrightarrow E_e(u^0)< \infty \; 
	\text{ as } \;  t_1 \rightarrow 0; 
\end{equation*}
hence \eqref{electrostaticE_2} also holds true for $t_1=0$.

Altogether we have verified that 
\begin{align*}
	E(u(t_2)) - E(u(t_1)) 
	& = E_m(u(t_2)) -E_m(u(t_1)) -
	\lambda (E_e(u(t_2)) - E_e(u(t_1)))
	\nonumber \\
	& = \int_{t_1}^{t_2} \int_D (\beta \Delta^2 u(s)- \tau 
	\Delta u(s) + \lambda g(u(s)))\partial_t u(s)\, \rd x\rd s 
	\nonumber \\
	& = - \int_{t_1}^{t_2} \Vert \partial_t u(s)\Vert_{L_2(D)}^2\, \rd s,
	\quad 0\leq t_1 \leq t_2 \leq T, 
\end{align*}
where in the last step we used equation \eqref{evolution1}, and this completes the proof of Theorem ~\ref{Tenergyequality}. 
\qed
\medskip

The energy equality \eqref{energyequality} provides a crucial step in the proof of the refined global existence criterion stated in Theorem \ref{globalcriterion}, which we shall discuss in the next section.

%%%%%%%%%%%%%%%%
%%%%%%%%%%%%%%%%
\section{Refined Criterion for Global Existence: Proof of Theorem \ref{globalcriterion}} %or: Improved criterion for global existence 
\label{rc}
%%%%%%%%%%%%%%%%
%%%%%%%%%%%%%%%%
In this section, we improve the global existence criterion in part (ii) of Theorem  \ref{wellposedness} by showing that $u$ cannot blow up in $W^{4\xi}_{2,\mathcal{B}}(D)$ in finite time, thus
touchdown of $u$ on the ground plate is the only possible finite time singularity. To this end, we need slightly more regularity on the boundary of $D$, namely, $\partial D \in C^{4,\gamma}$. 

The rest of this section is devoted to the proof of Theorem \ref{globalcriterion}. The proof follows
the lines of \cite{LaurencotWalker2018}, with some modifications to account for the hinged boundary 
conditions \eqref{evolution2}. Note that the second statement in the theorem is obtained by 
applying the first one to an arbitrary $T_0>0$. 
 
From now on, $(u, \psi_u)$ is the solution to \eqref{evolution1}-\eqref{evolution5} enjoying 
the regularity \eqref{regularity.1}-\eqref{regularity.2} and satisfying the lower bound \eqref{singul1}. We aim at proving that 
\begin{equation}
	\label{singul2}
	\Vert u(t) \Vert_{W^{4 \xi}_{2, \mathcal{B}}(D)} \leq c_2(\rho_0,T_0), \quad t \in [0,T_m)\cap [0,T_0],
\end{equation}
so that Theorem \ref{wellposedness} (ii) in turn yields Theorem \ref{globalcriterion}. 

Set for $(t,x) \in [0,T_m)\times D$
\begin{equation*}
	G(u(t))(x) := ( 1 + \varepsilon^2 \vert \nabla u(t,x)  \vert^2) (\partial_z \psi_{u(t)}(x,u(t,x)))^2
\end{equation*}
and recall that thanks to identity \eqref{kidentity}, the right-hand side of equation \eqref{evolution1}
equals to $- \lambda G(u(t))(x)$. Owing to \eqref{singul1} and the continuous embedding of 
$W^{4\xi}_2(D)$ in $W^2_3(D)$, Proposition \ref{ellipticp} ensures that $G(u(t))$ belongs to $L_2(D)$ for $t \in [0,T_m)\cap [0,T_0]$. Moreover, arguing exactly as in 
\cite[Corollary 3.5]{LaurencotWalker2018} ($v=u(t)$) it holds for $t \in [0,T_m)\cap [0,T_0]$
\begin{equation}
\label{coroA}
	\Vert G(u(t)) \Vert_{L_1(D)} \leq \bigg( 4 + \frac{2}{\rho_0^2}\bigg) \vert D \vert + 
	4 \varepsilon^2 \Vert \nabla u(t)\Vert^2_{L_2(D)} .
\end{equation}
Namely, the $L_1$-norm of $G(u(t))$ is controlled by the $H^1$-norm of $u(t)$. Noticing that the 
electrostatic energy $E_e$ defined in \eqref{energye} is exactly the same as in \cite{LaurencotWalker2018}, an 
application of \cite[Lemma 3.6]{LaurencotWalker2018} ($v=u(t)$) shows that 
\begin{align*}
	E(u(t))&= E_m(u(t))-\lambda E_e(u(t)) 
	\\
	&= E_m(u(t)) -\lambda \vert D \vert + 
	\lambda \int_D u(t) ( 1+ \varepsilon^2 \vert \nabla u(t)  \vert^2) \, \partial_z \psi_{u(t)}(\cdot, u(t))\, \rd x
\end{align*}
for $t \in [0,T_m)\cap [0,T_0]$. Estimating the last term on the right-hand side as in the 
proof of \cite[Corollary 3.7]{LaurencotWalker2018} ($v=u(t)$) gives the following lower bound on the total energy $E$: 
\begin{equation}
	\label{coroB}
	E(u(t)) \geq  E_m(u(t)) - 3 \lambda \varepsilon^2\Vert \nabla u(t)\Vert^2_{L_2(D)}  
	- \lambda  \bigg( 4 + \frac{1}{2\rho_0^2}\bigg) \vert D \vert	
\end{equation}
for $t \in [0,T_m)\cap [0,T_0]$. Henceforth we shall denote by $c_2$ a positive constant which 
may vary from line to line and depends only on $\rho_0$, $T_0$, $u^0$, $\beta$, 
$\varepsilon$, $\lambda$, and $D$. In particular, $c_2$ does not depend on $T_m$. 

We first establish an $L^2$-bound on $u(t)$. 

\begin{lemma}
	\label{lemmC}
	There is $c_2>0$ such that 
	\begin{equation*}
		\Vert u(t) \Vert_{L_2(D)} \leq c_2, \quad  t \in [0,T_m)\cap [0,T_0].
	\end{equation*}
\end{lemma}

\begin{proof}
Multiplying equation \eqref{evolution1} by $u(t)$ and integrating over D gives, for $t \in (0,T_m)$, 
\begin{equation*}
	%\label{lemmC_1}
\frac{1}{2}\frac{\rd}{\rd t} \int_D u(t)^2 \, \rd x +  \int_D  u(t) \big( \beta \Delta^2 u(t) - \tau  \Delta u(t) \big) \,  \rd x  = - \lambda \int_D  u(t) G(u(t)) \, \rd x.
\end{equation*}
Applying Green's formula twice and using the boundary conditions \eqref{evolution2} leads to 
\begin{align*}
	&\frac{1}{2}\frac{\rd}{\rd t} \int_D u(t)^2 \, \rd x +  \int_D   \big( \beta (\Delta u(t))^2 + \tau  \vert \nabla u(t)\vert^2  \big) \,  \rd x  - \beta (1-\sigma) \int_{\partial D} \kappa (\partial_{\nu} u(t))^2 \, \rd\omega 
	\\
	&\,= - \lambda \int_D  u(t) G(u(t)) \, \rd x,
\end{align*}
and \cite[Lemma A.1]{SweersVassi2018} further entails that 
\begin{equation*}
	\frac{1}{2}\frac{\rd}{\rd t} \int_D u(t)^2 \, \rd x +  2E_m(u(t))  
	= - \lambda \int_D  u(t) G(u(t)) \, \rd x, \quad t \in (0,T_m).
\end{equation*}
Hence, since $u(t)>-1$ and $G(u(t)) \geq 0$ in $D$, it follows that 
\begin{equation}
\label{lemmC_3}
\frac{1}{2} \frac{\rd}{\rd t} \Vert u(t)\Vert_{L_2(D)}^2 + 2 E_m(u(t))
 \leq \lambda \Vert G(u(t)) \Vert_{L_1(D)}, \quad t \in (0, T_m).
\end{equation}
On the other hand, from \eqref{coroA} and interpolation we infer that, for $t \in [0, T_m)\cap 
[0,T_0]$, 
\begin{equation}
	\label{lemmC_5}
	\Vert G(u(t)) \Vert_{L_1(D)} \leq c_2 \big(1 +\Vert \nabla u(t) \Vert^2_{L_2(D)} \big) 
	\leq c_2 \big( 1+ \Vert \Delta u(t) \Vert_{L_2(D)} \Vert u(t) \Vert_{L_2(D)}\big).
\end{equation}
Because the inequality 
\begin{equation*}
	\frac{1}{2} (\Delta u(t))^2 \leq \sum_{i,j=1}^2  (\partial_{x_i}\partial_{x_j} u(t))^2 \; \text{ in }\; D
\end{equation*}
implies 
\begin{align}
\label{lemmC_6}
	2 E_m(u(t))&= \beta \int_D \bigg(  (\Delta u(t))^2 + (1-\sigma) \Big[ \sum_{i,j=1}^2  (\partial_{x_i}\partial_{x_j} u(t))^2 -(\Delta u(t))^2  \Big]\bigg) \, \rd x 
	\nonumber \\
	& \quad + \tau \int_D \vert \nabla u(t) \vert^2 \, \rd x
	\nonumber \\
	& \geq \frac{\beta (1+\sigma)}{2}  \Vert \Delta u(t) \Vert^2_{L_2(D)}
\end{align}
for $t \in [0,T_m)$, we obtain from \eqref{lemmC_5} using Young's inequality 
\begin{equation}
\label{lemmC_7}
\Vert G(u(t)) \Vert_{L_1(D)} \leq \frac{1}{\lambda} E_m(u(t)) + c_2 \big(1 + \Vert  u(t) \Vert^2_{L_2(D)} \big), \quad t \in [0, T_m)\cap [0,T_0].
\end{equation}
Now, combining the inequalities \eqref{lemmC_3} and \eqref{lemmC_7} yields 
\begin{equation*}
	\frac{1}{2} \frac{\rd}{\rd t}  \Vert u(t)\Vert_{L_2(D)}^2  + E_m(u(t)) \leq c_2 \big(1 + 
	\Vert  u(t) \Vert^2_{L_2(D)} \big) , \quad t \in (0, T_m)\cap (0,T_0],
\end{equation*}
and hence $\tfrac{1}{2}  \tfrac{\rd}{\rd t} \Vert u(t)\Vert_{L_2(D)}^2 \leq c_2 \big (1 + 
\Vert  u(t) \Vert^2_{L_2(D)} \big)$. Solving this differential inequality, we conclude the assertion. 
\end{proof}

The next result shows that the mechanical energy is only controlled by the total energy. 
%the next result is a control of the mechanical energy iinvolving only the  total energy. 
\begin{lemma}
	\label{lemmD}
There is $c_2 >0$ such that 
	\begin{equation*}
	E(u(t)) \geq  \frac{1}{2}  E_m(u(t)) - c_2,  \quad t \in [0, T_m)\cap [0,T_0].
	\end{equation*}
\end{lemma}

\begin{proof}
The lower bound \eqref{coroB} on $E$ together with interpolation, \eqref{lemmC_6}, and 
Young's inequality gives 
\begin{align*}
%\label{lemmD_1}
E(u(t)) &\geq  E_m(u(t)) - c_2 E_m(u(t))^{\frac{1}{2}} \Vert u(t)\Vert_{L_2(D)} -c_2
\\
& \geq \frac{1}{2} E_m(u(t))- c_2 \Vert u(t)\Vert_{L_2(D)}^2 -c_2
\end{align*}
for $t \in [0, T_m)\cap [0,T_0]$. Lemma \ref{lemmC} finishes the proof.
\end{proof}

The energy equality \eqref{energyequality} now allows us to derive the following bound on the 
$L^2$-norm of ~$\Delta u(t)$. 

\begin{corollary}
	\label{coroE}
	There is $c_2>0$ such that 
	\begin{equation*}
		\frac{\beta (1+\sigma)}{8}  \Vert \Delta u(t)\Vert^2_{L_2(D)} + 
		\int_0^t \Vert \partial_t u(s)\Vert^2_{L_2(D)} \, \rd s \leq c_2, \quad t \in [0, T_m)\cap [0,T_0].
	\end{equation*}
\end{corollary}

\begin{proof}
	According to Theorem \ref{Tenergyequality}, we have 
	\begin{equation*}
		E(u(t)) + \int_0^t \Vert \partial_t u(s)\Vert_{L_2(D)}^2 \, ds = E(u^0), \quad t \in [0,T_m), 
	\end{equation*}
and by Lemma \ref{lemmD}, 
	\begin{equation*}
		E(u^0) \geq \frac{1}{2} E_m(u(t)) - c_2 + 
		\int_0^t \Vert \partial_t u(s)\Vert_{L_2(D)}^2 \, \rd s , \quad t \in [0, T_m)\cap [0,T_0].
	\end{equation*}
The assertion follows from \eqref{lemmC_6} and the fact that $E(u^0) < \infty$.
\end{proof}

Putting together the above results we finally obtain an upper bound on the $L^1$-norm of the 
right-hand side of \eqref{evolution1}. 

\begin{corollary}
	\label{coroF}
	There is $c_2>0$ such that 
	\begin{equation*}
		\Vert G(u(t)) \Vert_{L_1(D)} \leq c_2, \quad t \in [0, T_m)\cap [0,T_0].
	\end{equation*}
\end{corollary}

\begin{proof}
By \cite[Theorem 3.1.2.1]{Grisvard1985} (since $D$ is convex) and Corollary \ref{coroE}, we deduce that 
\begin{equation*}
	\Vert u(t) \Vert^2_{W^2_2(D)} \leq c_2, \quad t  \in [0, T_m)\cap [0,T_0].
\end{equation*}
Now \eqref{coroA} yields the assertion.
\end{proof}

All that is left to prove is that the bound from Corollary \ref{coroF} %on the right-hand side of \eqref{evolution1} 
implies a bound on 
%the solution
 $u(t)$ in $W^{4\xi}_{2,\mathcal{B}}(D)$, i.e. inequality \eqref{singul2}. 
 
 So far we have not used the higher regularity of the boundary $\partial D$, but we do so now. %but this somes next.
 Recalling that $\partial D \in C^{4,\gamma}$, we introduce $B^s_{1,1,\mathcal{B}}(D)$ for 
 $s\in (-3-\gamma, 4+\gamma)\backslash\{1,3\}$, i.e. the Besov space $B^s_{1,1}(D)$ incorporating 
 the boundary conditions \eqref{evolution2}, if meaningful:
 \begin{equation*}
 	B^{s}_{1,1,\mathcal{B}}(D):= 
 	\left\{\begin{array}{ll}
 		\vspace{0.1cm}
 		B^{s}_{1,1}(D),  &s\in (-3-\gamma,1),\\ 
 		\vspace{0.1cm}
 		\left\{ v\in  B^{s}_{1,1}(D)\, : \, v=0 \text{ on } \partial D\right\},
 		&s \in (1,3),\\
 		\left\{v\in  B^{s}_{1,1}(D)\, : \, v = \Delta v -(1 - \sigma) \kappa
 		\partial_{\nu}v=0 \text{ on } \partial D\right\},
 		&s\in (3,4+\gamma) .
 	\end{array} \right.
 \end{equation*} 
From now on, we assume $4\xi \in ( \tfrac{5}{2}, \tfrac{7}{2})$ and fix $\alpha \in (4\xi - \tfrac{7}{2},0)$.  The cases of 
$4\xi \in ( \tfrac{7}{3}, \tfrac{5}{2})$ and $4\xi \in (\tfrac{7}{2}, 4)$ are treated the same way. The constant $c_2$ is now allowed to depend also on $\xi$ and $\alpha$ (but still not on $T_m$); dependence on additional parameters is explicitly indicated. 

Let us first check that the operator $-A$, given by 
\begin{equation*}
	-A v:= -(\beta \Delta^2  -\tau \Delta )v, \quad v \in B^{4+\alpha}_{1,1,\mathcal{B}}(D),
\end{equation*}
generates a strongly continuous analytic semigroup ~$\{e^{-tA} \, : \, t\geq 0\}$ on 
$B^{\alpha}_{1,1}(D)$ satisfying the regularizing property stated in Lemma \ref{lemmG} below. 
Note that in Section \ref{wp} we have already shown that $-A$ restricted to $W^{4}_{2,\mathcal{B}}(D)$ generates a strongly continuous analytic 
semigroup ~$\{e^{-tA} \, : \, t\geq 0\}$ on $L_2(D)$ with 
\begin{equation}
\label{lemmG_0}
	\Vert e^{-tA} \Vert_{\mathcal{L}(W^{4\xi}_{2,\mathcal{B}}(D))} 
	\leq M,  \quad t\geq 0. 
\end{equation}
Arguing in a similar fashion we obtain the following result.  

\begin{lemma}
	\label{lemmG}
	It holds that 
	\begin{equation}
		\label{lemmG_1}
		A \in \mathcal{H}(B^{4+\alpha}_{1,1,B}(D),B^{\alpha}_{1,1}(D)). 
	\end{equation} 
	Moreover, given $\theta \in (0,1)$ with $\theta \notin \left\{ \tfrac{(1-\alpha)}{4}, \tfrac{(3-\alpha)}{4}\right\}$, 
	there is  $c_2(\theta)>0$  such that, for $t\in (0,T_0]$, 
	\begin{equation}
		\label{lemmG_2}
		\Vert e^{-tA} \Vert_{\mathcal{L}(B^{\alpha}_{1,1}(D), 
			B^{4\theta+\alpha}_{1,1,\mathcal{B}(D)})} \leq c_2(\theta) t^{-\theta}.
	\end{equation}
\end{lemma}

\begin{proof}
To prove \eqref{lemmG_1}, we shall apply  \cite[Theorem 2.18]{Guidetti1989}. Note that $\alpha \in (4\xi - \tfrac{7}{2},0) \subset (-1,1)$. Let us verify that assumptions $(m)$, $(n)$, and $(o)$ of \cite[Theorem 2.18]{Guidetti1989} are satisfied. Assumptions  $(m)$ and $(n)$ follow identically 
to the proof of Lemma \ref{aSemigroup} (see \eqref{prsymbolA} and \eqref{boundaryOperator}). 
Assumption $(o)$ requires that, for any $x \in \partial D$, $\zeta \in \mathbb{R}^2$, $r\geq 0$ 
with $\zeta \cdot \nu(x)=0$ and $(\zeta, r) \neq (0,0)$, and any $\vartheta \in [-\tfrac{\pi}{2}, \tfrac{\pi}{2}]$, 
zero is the only bounded solution in $[0, \infty)$ to 
\begin{empheq}{align*}
\big(- \beta (\vert \zeta \vert^2 - \partial_t^2 )^2 - r e^{i\vartheta}\big)v &=0, 
	 \\%[-0.05cm] 
	v(0) = \partial_t^2 v(0) &=0.
\end{empheq}
This has also been proved in Lemma \ref{aSemigroup} (see \eqref{Lopa.1}-\eqref{Lopa.2}). 
Consequently, we are able to apply \cite[Theorem 2.18]{Guidetti1989}. 

It remains to prove \eqref{lemmG_2}. From \cite[Proposition 4.13]{Guidetti1990} we infer that 
\begin{equation*}
	\bigl(B^{\alpha}_{1,1}(D),B^{4+\alpha}_{1,1,\mathcal{B}}(D)\bigr)_{\theta,1} 
	\doteq B^{4\theta+\alpha}_{1,1,\mathcal{B}}(D), \quad 4\theta \in (0,4) \backslash \{1-\alpha, 3-\alpha\},
\end{equation*}
where $(\cdot,\cdot)_{\theta,1}$ is the real interpolation functor. Then \cite[Lemma II.5.1.3]{Amann1995} implies \eqref{lemmG_2}.
\end{proof}
\medskip

We are now ready to prove Theorem \ref{globalcriterion}. \\

\noindent{\textit{Proof of Theorem \ref{globalcriterion}}.} First, we have the following conitinuous 
embeddings, due to ~\cite[Section 5]{Amann1993} (or ~\cite[Section 4]{Guidetti1993}): 
\begin{equation*}
	B^{4+\alpha}_{1,1,\mathcal{B}}(D) \hookrightarrow B^{s}_{1,1,\mathcal{B}}(D) 
	\hookrightarrow B^{0}_{1,1}(D) \hookrightarrow L_1(D) \hookrightarrow B^{0}_{1,\infty}(D)
	\hookrightarrow B^{\alpha}_{1,1}(D), \; s \in (0,4+\alpha)\backslash\{1,3\}, 
\end{equation*}
which, together with Corollary \ref{coroF}, yield 
\begin{equation}
	\label{refinedcriterion_1}
	\Vert G(u(t)) \Vert_{B^{\alpha}_{1,1}(D)} \leq c_2, \quad t \in [0, T_m)\cap [0,T_0].
\end{equation}
We next fix $\theta \in (0,1)$ and $4\xi_1\in (4\xi,4)\backslash\{3\}$ so that 
\begin{equation*}
	4\theta + \alpha > 4 \xi_1 +1 > 4\xi +1 
\end{equation*}
and hence, by \cite[Section 5]{Amann1993}, we get
\begin{equation}
	\label{refinedcriterion_2}
	B^{4\theta+\alpha}_{1,1,\mathcal{B}}(D) \hookrightarrow B^{4\xi_1}_{2,2,\mathcal{B}}(D) 
	\doteq W^{4\xi_1}_{2,\mathcal{B}}(D) \hookrightarrow W^{4\xi}_{2,\mathcal{B}}(D).
\end{equation}
Using the variation of constants formula
\begin{equation*}
	u(t)=e^{-tA} u^0 - \lambda \int_0^t e^{-(t-s)A} G(u(s)) \, \rd s,\quad t\in [0,T_m),
\end{equation*}
we derive from \eqref{lemmG_0}, \eqref{lemmG_2}, \eqref{refinedcriterion_1}, and  \eqref{refinedcriterion_2} that
\begin{align*}
\Vert u(t)\Vert_{W^{4\xi}_{2,\mathcal{B}}(D)} 
	&\leq \Vert e^{-tA}u^0\Vert_{W^{4\xi}_{2,\mathcal{B}}(D)} + \lambda \int_0^t 
	\Vert e^{-(t-s)A} G(u(s))\Vert_{W^{4\xi}_{2,\mathcal{B}}(D)} \, \rd s
	\nonumber\\
	& \leq \Vert e^{-tA}\Vert_{\mathcal{L}(W^{4\xi}_{2,\mathcal{B}}(D))} 
	\Vert u^0\Vert_{W^{4\xi}_{2,B}(D)} +\lambda c_{B,W}
	\int_0^t \Vert e^{-(t-s)A}G(u(s))\Vert_{B^{4\theta+\alpha}_{1,1,\mathcal{B}}(D)} \, \rd s
	\nonumber\\
	& \leq  M \Vert u^0\Vert_{W^{4\xi}_{2,\mathcal{B}}(D)} + \lambda c_{B,W} \int_0^t 
	\Vert e^{-(t-s)A} \Vert_{\mathcal{L}(B^{\alpha}_{1,1}(D), B^{4\theta+\alpha}_{1,1,\mathcal{B}}(D))} \Vert G(u(t)) \Vert_{B^{\alpha}_{1,1}(D)}
	\, \rd s 
	\nonumber\\
	&\leq c_2(\theta)
\end{align*}
for $t\in [0,T_m)\cap [0,T_0]$, where $c_{B,W}>0$ is the corresponding embedding constant. So we have shown \eqref{singul2} and Theorem \ref{wellposedness} (ii) finishes 
the proof of Theorem \ref{globalcriterion}.
\qed
\medskip

It is an open problem to show that if $\lambda$ is sufficiently large, then $T_m$ must be finite. 
%%%%%%%%%%%%%%%%
%%%%%%%%%%%%%%%%
\section*{Acknowledgements}
This paper is an edited extract of the author's Ph.D. thesis submitted to the Leibniz Universität Hannover. The author gratefully acknowledges the support of her thesis supervisor Prof. Christoph Walker. 
%%%%%%%%%%%%%%%%
%%%%%%%%%%%%%%%%
\bibliographystyle{siam}
\bibliography{HingedParabolic}
%%%%%%%%%%%%%%%%
%%%%%%%%%%%%%%%%

%%%%%%%%%%%%%%%%
%%%%%%%%%%%%%%%%
\end{document}